\documentclass[english]{article}
\usepackage[T1]{fontenc}
\usepackage[utf8]{inputenc}
\usepackage{geometry}
\usepackage{color}
\usepackage{babel}
\usepackage{varioref}
\usepackage{units}
\usepackage{enumitem}               	
\setlist[enumerate,1]{label=\roman*)}  

\usepackage{amsmath, amsthm, amssymb}
\usepackage{graphicx}
\usepackage[numbers]{natbib}
\usepackage[unicode=true, bookmarks=true, bookmarksnumbered=false, bookmarksopen=false, breaklinks=false, pdfborder={0 0 0}, pdfborderstyle={}, colorlinks=true, linkcolor=blue, citecolor= blue]{hyperref}
\usepackage{microtype}                  

\usepackage{mathrsfs} 
\usepackage{todonotes}
\usepackage{dsfont}             

\theoremstyle{plain}
  \newtheorem{thm}{\protect\theoremname}
  \theoremstyle{remark}
  \newtheorem{rem}[thm]{\protect\remarkname}
  \theoremstyle{definition}
  \newtheorem{defn}[thm]{\protect\definitionname}
  \theoremstyle{plain}
  \newtheorem{prop}[thm]{\protect\propositionname}
  \theoremstyle{definition}
  
  \theoremstyle{plain}
  \newtheorem{cor}[thm]{\protect\corollaryname}
  \theoremstyle{plain}
  \newtheorem{lem}[thm]{\protect\lemmaname}

\newcommand{\AVaR}{{\sf AV@R}}

\DeclareMathOperator{\1}{{\mathds 1}} 		
\newcommand{\sign}{\mbox{sign}\,} 

  \providecommand{\corollaryname}{Corollary}
  \providecommand{\definitionname}{Definition}
  \providecommand{\examplename}{Example}
  \providecommand{\lemmaname}{Lemma}
  \providecommand{\propositionname}{Proposition}
  \providecommand{\remarkname}{Remark}
\providecommand{\theoremname}{Theorem}

\begin{document}

\title{On Banach spaces of vector-valued random variables and their duals motivated by risk measures}

\author{Thomas Kalmes \and Alois Pichler\footnote{Faculty of Mathematics, Chemnitz University of Technology, Germany (both authors) thomas.kalmes@mathematik.tu-chemnitz.de; alois.pichler@mathematik.tu-chemnitz.de}}

\maketitle

\begin{abstract}
We introduce Banach spaces of vector-valued random variables motivated from mathematical finance. So-called risk functionals are defined in a natural way on these Banach spaces and it is shown that these functionals are Lipschitz continuous. The risk functionals cannot be defined on strictly larger spaces of random variables which creates a particular interest for the spaces presented. We elaborate key properties of these Banach spaces and give representations of their dual spaces in terms of vector measures with values in the dual space of the state space.

\textbf{Keywords:} Vector-valued random variables, Banach spaces of random variables, rearrangement invariant spaces, dual representation, risk measures, stochastic dominance

\textbf{2010 Mathematics Subject Classification:} Primary: 46E30, 46E40. Secondary: 62P05
\end{abstract}

\section{Introduction}

This paper introduces Banach spaces for vector-valued random variables
in a first part. These spaces extend rearrangement spaces for functions
in two ways. First, random variables are considered on a probability
space and second, we extend them to vector-valued (i.e., $\mathbb{R}^{d}$,
or more general Banach space-valued) random variables. 

It is natural to address 
differences/\,similarities between $L^{1}$ and $L^{p}$ spaces and we elaborate on extensions in the second part of the paper. We fully describe the duals of the new spaces. 
The duality theory for these spaces differs essentially from $L^p$ spaces. The new
spaces are larger than $L^\infty$, but not an $L^{p}$ space in general
and further, their dual is not even similar to $L^{p}$ spaces. However,
they are reflexive. The duality theory is particularly nice in case
that the dual of the state space enjoys the Radon\textendash Nikodým property.

\medskip{}

An important motivation for considering these spaces derives from
recent developments in mathematical finance. Vector-valued functions
or portfolio vectors are naturally present in many real life situations.
An example is given by considering a portfolio with investments in
$d$, say, different currencies. The random outcome is in $\mathbb{R}^{d}$
in this motivating example, the related random variable is said to be vector-valued. Here, we consider more generally Banach space-valued random
variables. The spaces can be associated with risk functionals and
we demonstrate that the spaces introduced are as large as possible
such that the associated risk functionals remain continuous.



\medskip{}

\citet{Rueschendorf2006} introduces and considers vector-valued risk functionals first.  \citet{Svindland2009,Filipovic2012,SvindlandKupper2011} and many
further authors consider and discuss different domain spaces for risk
measures on portfolio vectors, for example Orlicz spaces (as done
in\ \citet{Cheridito2009a} and \citet{Bellini2012}). \citet{Ekeland2011}
consider the domain space $L^{\infty}$ for these risk measures. 
\citet{EkelandGalichonHenry} provide the first multivariate generalization of a Kusuoka representation for risk measures on vector-valued random variables on $L^{2}$. 
In contrast, the present paper extends these spaces and presents
the largest possible Banach spaces for which those functionals remain
continuous. The resulting spaces are neither Orlicz nor Lebesgue spaces,
as considered in the earlier literature.

The spaces, which we consider, are in a way related to function
spaces (rearrangement spaces) introduced by \citet{LorentzSpaces,Lorentz},
following earlier results obtained by \citet{Halperin}. 
For unexplained notions from the theory of vector measures we would like to refer the reader to the book by \citet{Diestel1977}.

\paragraph{Outline of the paper.}

The following section (Section~\ref{sec:Relation}) provides the
mathematical setting including the relation to mathematical finance. The Banach spaces $L_{\sigma}^{p}(P,X)$ of $X$-valued random variables, introduced in Section~\ref{sec:VectorValued}, constitute the natural domains of risk functionals. We demonstrate that risk functionals are continuous with respect to the norm of the space introduced. In Section~\ref{sec:ScalarDual} we give a representation of the dual spaces of these Banach spaces in the scalar-valued case. This representation
is used in Section~\ref{sec:VectorDual} to derive representations of the duals in the general vector-valued case.

\section{\label{sec:Relation}Mathematical setting and motivation }

We consider a probability space $\left(\Omega,\mathcal{F},P\right)$
and denote the \emph{distribution function} of an $\mathbb{R}$\nobreakdash-valued
random variable $Y$ by 
\[
F_{Y}(q):=P\left(Y\le q\right)=P\left(\left\{ \omega:\,Y(\omega)\le q\right\} \right).
\]
The \emph{generalized inverse} is the nondecreasing and left-continuous
function 
\[
F_{Y}^{-1}(\alpha):=\inf\left\{ q:\,P\left(Y\le q\right)\ge\alpha\right\} ,
\]
also called the \emph{quantile} or \emph{Value-at-Risk}.

With $X=(X,\Vert\cdot\Vert)$ we denote a Banach space and by $X^*$ its continuous dual space. We use the notation $\langle\varphi,x\rangle$ for $\varphi(x)$, $\varphi\in X^*$ and $x\in X$. As usual we denote for $p\in [1,\infty)$ by $L^p(P,X)$ the Bochner-Lebesgue space of $p$-Bochner integrable $X$-valued random variables $Y$ on $(\Omega,\mathcal{F},P)$ whose norm we denote by $\Vert\cdot\Vert_p$. Recall that for $Y\in L^p(P,X)$
\begin{equation}
\left\Vert Y\right\Vert _{p}=\left(\int_{0}^{1}F_{\left\Vert Y\right\Vert }^{-1}(u)^{p}\mathrm{d}u\right)^{1/p}=\left(\int_{0}^{\infty}p\,t^{p-1}\left(1-F_{\left\Vert Y\right\Vert }(t)\right)\mathrm{d}t\right)^{1/p}.\label{eq:Lp}
\end{equation}

\medskip{}

In this paper Banach spaces of vector-valued, strongly
measurable random variables are introduced by weighting the quantiles in a different
way than~\eqref{eq:Lp}. The present results extend and generalize
characterizations obtained in \citet{Pichler2013a}, where only real
valued random variables and $p=1$ are considered (and elaborated
in a context of insurance).
\begin{rem}\label{rem:Extension}
	We shall assume throughout the paper that the
	probability space $(\Omega,\mathcal{F},P)$ is rich enough to carry a $[0,1]$-valued, uniform distribution.\footnote{$U$ is uniform, if $P(U\le u)=u$ for all $u\in[0,1]$.}
	If this is not the case, then one may replace $\Omega$ by $\tilde{\Omega}:=\Omega\times[0,1]$
	with the product measure $\tilde{P}(A\times B):=P(A)\cdot\text{Lebesgue measure}(B)$.
	Every random variable $Y$ on $\Omega$ extends to $\tilde{\Omega}$
	by $\tilde{Y}(\omega,u):=Y(\omega)$ and $U(\omega,u):=u$ is a uniform random variable, as $\tilde{P}\left(U\le u\right)=\tilde{P}\left(\Omega\times[0,u]\right)=u$. We denote the set of $[0,1]$-valued uniform random variables on $(\Omega,\mathcal{F},P)$ by $\mathscr{U}(0,1)$.
	
	With an $\mathbb{R}$\nobreakdash-valued random variable $Y$ one
	may further associate its \emph{generalized quantile} transform 
	\[
	F(y,u):=(1-u)\cdot\lim_{y^{\prime}\uparrow y}F_{Y}(y^{\prime})+u\cdot F_{Y}(y).
	\]
	The random variable $F(Y,U)$ is uniformly distributed again
	and $F(Y,U)$ is coupled in a comonotone way with $Y$, i.e.,  
	the inequality  $\big(F(Y,U)(\omega)-F(Y,U)(\omega^\prime)\big)\big(Y(\omega)-Y(\omega^\prime)\big)\geq 0$ holds $P\otimes P$ almost everywhere (see, e.g., \citet[Proposition 1.3]{PflugRomisch2007}).
\end{rem}

\subsection*{Relation to mathematical finance: risk measures and their continuity
	properties.}

Risk measures on $\mathbb{R}$\nobreakdash-valued random variables
have been introduced in the pioneering paper~\citet{Artzner1999}.
An $\mathbb{R}$\nobreakdash-valued random variable is typically associated
with the total, or accumulated return of a portfolio in mathematical
finance. (The prevalent interpretation in insurance is the size of
a claim, which happens with a probability specified by the probability
measure $P$.)

The aggregated portfolio is composed of individual components, as
stocks. From the perspective of comprehensive risk management it is
desirable to understand not only the risk of the accumulated portfolio,
but also its components. These more general risk measures on $\mathbb{R}^{d}$\nobreakdash-valued
random variables have been considered first in \citet{BurgertRueschendorf}
and further progress was made, for example, by \citet{Rueschendorf2006},
\citet{EkelandGalichonHenry} and \citet{Ekeland2011}.

\citet[Theorem 1.7]{Ekeland2011} obtain a Kusuoka representation (cf.\ \citet{Kusuoka}) for risk measures
based on $\mathbb{R}^{d}$\nobreakdash-valued random variables. 
The risk functional
identified there in the ``regular case'' for the homogeneous risk
functional on random vectors is 
\begin{equation}
\rho_{Z}(Y):=\sup\left\{ \mathbb{E}\left\langle Z,Y^{\prime}\right\rangle \colon Y^{\prime}\sim Y\right\} ,\label{eq:3}
\end{equation}
where $Y\sim Y^{\prime}$ indicates that $Y$ and $Y^{\prime}$ enjoy
the same law in $\mathbb{R}^{d}$.\footnote{That is, $P\left(Y_{1}\le y_{1},\dots Y_{d},\le y_{d}\right)=P\left(Y_{1}^{\prime}\le y_{1},\dots ,Y_{d}^{\prime}\le y_{d}\right)$
	for all $(y_{1},\dots y_{d})\in\mathbb{R}^{d}$.} $\rho_{Z}$ is called the \emph{maximal correlation risk measure
	in direction $Z$}.


The rearrangement inequality (see e.g.\ \citet[Theorem 5.25(2)]{McNeil2005}, also known as Chebyshev's sum inequality, cf.\ \citet[Section~2.17]{Hardy1988}) provides an upper bound for the natural linear form in~\eqref{eq:3} by 
\begin{equation}
\left|\mathbb{E}\left\langle Z,Y\right\rangle \right|\le\mathbb{E}\left\Vert Z\right\Vert ^{*}\cdot\left\Vert Y\right\Vert \le\mathbb{E}\,K\cdot\left\Vert Z\right\Vert _{\ell_{1}^{d}}\cdot\left\Vert Y\right\Vert \le K\cdot\int_{0}^{1}F_{\left\Vert Z\right\Vert _{\ell_{1}^{d}}}^{-1}(u)\cdot F_{\left\Vert Y\right\Vert }^{-1}(u)\mathrm{d}u,\label{eq:5}
\end{equation}
where the norms $\left\Vert \cdot\right\Vert $ and $\left\Vert \cdot\right\Vert ^{*}$
are dual to each other on $\mathbb{R}^{d}$ (here, $K>0$ is the constant
linking the norms by $\left\Vert \cdot\right\Vert ^{*}\le K\cdot\left\Vert \cdot\right\Vert _{\ell_{1}^{d}}$
on (the dual of) $\mathbb{R}^{d}$). 

\medskip{}

The maximal correlation risk measure~\eqref{eq:3} employs the linear
form $\mathbb{E}\left\langle Z,Y\right\rangle $, which satisfies
the bounds~\eqref{eq:5}. This motivates fixing the function 
\begin{equation}
\sigma(\cdot):=F_{\left\Vert Z\right\Vert _{\ell_{1}^{d}}}^{-1}(\cdot)\label{eq:5-2}
\end{equation}
and to consider an appropriate vector space of random variables endowed with 
\[
\left\Vert Y\right\Vert _{\sigma}:=\int_{0}^{1}\sigma(u)\cdot F_{\left\Vert Y\right\Vert }^{-1}(u)\mathrm{d}u,
\]
(cf.~\citet{Pichler2017, PichlerAmbiguity, AhmadiPichler}).
It turns out that $\left\Vert \cdot\right\Vert _{\sigma}$ is a norm
(Theorem~\ref{Banach space} below) on this vector space of random variables and that the maximal correlation risk measure is continuous with respect to the norm (Proposition~\ref{continuity of risk measure}).

\section{\label{sec:VectorValued}The vector-valued Banach spaces $L_{\sigma}^{p}(P,X)$}

Motivated by the observations made in the previous section we introduce the following notions.

\begin{defn}\label{def:Distortion}A nondecreasing, nonnegative function $\sigma\colon\left[0,1\right)\to\left[0,\infty\right)$,
	which is continuous from the left and normalized by $\int_{0}^{1}\sigma(u)\mathrm{d}u=1$, is called
	a \emph{distortion function} (in the literature occasionally also\emph{
		spectrum function}, cf.\ \citet{Acerbi2002a}). 
\end{defn}

\begin{defn}
	For a distortion function $\sigma$, a Banach space $(X,\left\Vert \cdot\right\Vert)$ and a probability space $(\Omega,\mathcal{F},P)$ we define for $p\in[1,\infty)$ and a strongly measurable $X$-valued random variable $Y$ on $(\Omega,\mathcal{F},P)$
	\begin{equation*}
	\left\Vert Y\right\Vert_{\sigma,p}^p:=\sup_{U\text{ uniform}}\mathbb{E}\,\sigma(U)\left\Vert Y\right\Vert ^{p}=\sup_{U\text{ uniform}}\int_\Omega \sigma(U(\omega))\Vert Y(\omega)\Vert^p dP(\omega),
	\end{equation*}
	where the supremum is taken over all $U\in\mathscr{U}(0,1)$, i.e., over all $[0,1]$-valued, uniformly distributed random variables $U$ on $(\Omega,\mathcal{F},P)$. Moreover, we set
	\begin{equation*}
	L_\sigma^p(P,X):=\{Y:\Omega\rightarrow X\text{ strongly measurable and }\Vert Y\Vert^p_{\sigma,p}<\infty\},
	\end{equation*}
	where as usual we identify $X$-valued random variables which coincide $P$-almost everywhere.
\end{defn}

Obviously, for $\sigma=1$ one obtains the classical Bochner-Lebesgue spaces $L^p(P,X)$ which are well-known to be Banach spaces.

\begin{thm}\label{Banach space}
	$L_\sigma^p(P,X)$ is a vector space and $\Vert\cdot\Vert_{p,\sigma}$ is a norm on $L_\sigma^p(P,X)$ turning it into a Banach space which embeds contractively into $L^p(P,X)$.
	
	Moreover, for each $X$-valued, strongly measurable $Y$ on $(\Omega,\mathcal{F},P)$ and every $U\in\mathscr{U}(0,1)$ which is coupled in comonotone way with $\Vert Y\Vert$ it follows that
	\begin{equation}\label{alternative expression norm}
	\Vert Y\Vert_{\sigma,p}^p=\mathbb{E}\,(\sigma(U)\Vert Y\Vert^p)=\int_0^1 \sigma(u) F_{\Vert Y\Vert}^{-1}(u)^pdu.
	\end{equation}
\end{thm}

\begin{proof}
	We denote the probability measure on $(\Omega,\mathcal{F})$ with $P$-density $\sigma\circ U$ for some $U\in\mathscr{U}(0,1)$ by $\sigma(U)P$ and the expectation of a non-negative random variable $Z$ on $(\Omega,\mathcal{F},\sigma(U)P)$ by $\mathbb{E}_U(Z)$. We obviously have
	\begin{equation*}
	\Vert Y\Vert_{\sigma,p}^p=\sup_{U\in\mathscr{U}(0,1)}\mathbb{E}_U\Vert Y\Vert^p,
	\end{equation*}
	which implies that $L_\sigma^p(P,X)$ is a subspace of the intersection of Banach spaces $\bigcap_{U\in\mathscr{U}(0,1)}L^p(\sigma(U)P,X)$ and that $\Vert\cdot\Vert_{\sigma, p}$ is a seminorm on $L_\sigma^p(P,X)$.
	
	By the rearrangement inequality (see, e.g., \citet[Theorem 5.25(2)]{McNeil2005}), the well-known fact that $F_{\sigma(U)}^{-1}=\sigma$ and $(F_{\Vert Y\Vert}^{-1})^p=F_{\Vert Y\Vert^p}^{-1}$ it follows for every $U\in\mathscr{U}(0,1)$ and each $X$-valued, strongly measurable $Y$ on $(\Omega,\mathcal{F},P)$, that
	\begin{equation*}
	\mathbb{E}\,(\sigma(U)\Vert Y\Vert^p)\leq \int_0^1\sigma(u)F_{\Vert Y\Vert}^{-1}(u)^p du
	\end{equation*}
	so that
	\begin{equation}\label{upper bound for norm}
	\Vert Y\Vert_{\sigma,p}^p\leq \int_0^1 \sigma(u)F_{\Vert Y\Vert}^{-1}(u)^p du.
	\end{equation}
	Moreover, if we fix for an $X$-valued, strongly measurable $Y$ on $(\Omega,\mathcal{F},P)$ some $U\in\mathscr{U}(0,1)$ such that $U$ and $\Vert Y\Vert$ are coupled in a comonotone way (such $U$ exists due to our general assumption on $(\Omega,\mathcal{F},P)$ made in Remark~\ref{rem:Extension}) then (\citet{Kusuoka})
	\begin{equation*}
	\mathbb{E}\,\big(\sigma(U)\Vert Y\Vert^p\big)=\int_0^1\sigma(u)F_{\Vert Y\Vert}^{-1}(u)^p du.
	\end{equation*}
	Together with~\eqref{upper bound for norm} we obtain for each $X$-valued, strongly measurable $Y$ on $(\Omega,\mathcal{F},P)$ that there is $U\in\mathscr{U}(0,1)$ such that
	\begin{equation*}
	\Vert Y\Vert_{\sigma,p}^p=\mathbb{E}\,(\sigma(U)\Vert Y\Vert^p)=\int_0^1\sigma(u)F_{\Vert Y\Vert}^{-1}(u)^p du,
	\end{equation*}
	proving~\eqref{alternative expression norm}.
	
	In order to see that the seminorm $\Vert\cdot\Vert_{\sigma,p}$ on $L_\sigma^p(P,X)$ is in fact a norm we apply the continuous version of Chebychev's inequality (see, e.g., \citet[Eq.~12.314]{Gradsteyn2007}) to the nonnegative, nondecreasing functions $\sigma$ and $(F_{\Vert Y\Vert}^{-1})^p$ on $[0,1)$ to obtain
	\begin{equation*}
	\int_0^1\sigma(u)F_{\Vert Y\Vert}^{-1}(u)^p du\geq \int_0^1\sigma(u)du\cdot\int_0^1 F_{\Vert Y\Vert}^{-1}(u)^p du=\int_0^1 F_{\Vert Y\Vert}^{-1}(u)^p du=\mathbb{E}\,(\Vert Y\Vert^p),
	\end{equation*}
	where the last equality follows from $(F_{\Vert Y\Vert}^{-1})^p=F_{\Vert Y\Vert^p}^{-1}$. 
	In particular, together with~\eqref{alternative expression norm} we obtain for every $X$-valued, strongly measurable $Y$
	\begin{equation*}
	\mathbb{E}\,(\Vert Y\Vert^p)\leq \Vert Y\Vert_{\sigma,p}^p,
	\end{equation*}
	which proves that $L_{\sigma}^p(P,X)$ embeds contractively into $L^p(P,X)$ and that $\Vert Y\Vert_{\sigma,p}=0$ implies $Y=0$ so that $\Vert \cdot\Vert_{\sigma,p}$ is indeed a norm.
	
	Finally, in order to prove that $L_\sigma^p(P,X)$ is a Banach space when equipped with the norm $\Vert\cdot\Vert_{\sigma,p}$, we first note that a Cauchy sequence $(Y_n)_{n\in\mathbb{N}}$ in $L_\sigma^p(P,X)$ is also a Cauchy sequence in $L^p(P,X)$ so that there is $Y\in L^p(P,X)$ with $Y=\lim_{n\rightarrow\infty}Y_n$ in $L^p(P,X)$. From this we conclude that $Y=\lim_{k\rightarrow\infty}Y_{n_k}$ $P$-almost everywhere on $\Omega$ for some subsequence $(Y_{n_k})_{k\in\mathbb{N}}$ of $(Y_n)_{n\in\mathbb{N}}$. Since for each $\varepsilon>0$ there is $N\in\mathbb{N}$ such that for all $U\in\mathscr{U}(0,1)$
	\begin{equation*}
	\varepsilon^p>\mathbb{E}\,(\sigma(U)\Vert Y_n-Y_m\Vert^p)
	\end{equation*}
	whenever $n,m\geq N$ it follows with Fatou's Lemma that for every $U\in\mathscr{U}(0,1)$ and each $n\geq N$ we have
	\begin{equation*}
	\mathbb{E}\,(\sigma(U)\Vert Y-Y_n\Vert^p)=\mathbb{E}\,(\lim_{k\rightarrow\infty}\sigma(U)\Vert Y_{n_k}-Y_n\Vert^p)\leq\liminf_{k\rightarrow\infty}\mathbb{E}\,(\sigma(U)\Vert Y_{n_k}-Y_n\Vert^p)\leq \varepsilon^p,
	\end{equation*}
	i.e., $\Vert Y-Y_n\Vert_{\sigma,p}\leq \varepsilon^p$ for every $n\geq N$. Thus, we conclude that
	\[Y=(Y-Y_N)+Y_N\in L_\sigma^p(P,X)\]
	and that $(Y_n)_{n\in\mathbb{N}}$ converges to $Y$ in $L_\sigma^p(P,X)$.
\end{proof}

\begin{rem}
	By~\eqref{alternative expression norm} $L_\sigma^p(P,X)$-membership of $Y$ only depends on the quantile function $F_{\Vert Y\Vert}^{-1}$ so that $L_\sigma^p(P,X)$ is invariant with respect to rearrangements. From the definition of $\Vert\cdot\Vert_{\sigma,p}$ it follows immediately that $L_\sigma^p(P,X)$ is an $L^\infty(P)$-module and that $\Vert\alpha Y\Vert_{\sigma,p}\leq \Vert\alpha\Vert_\infty\,\Vert Y\Vert_{\sigma,p}$ for all $\alpha\in L^\infty(P)$ and each $Y\in L_\sigma^p(P,X)$.
\end{rem}

We next show that the $L_\sigma^p(P,X)$-spaces behave like the classical Bochner-Lebesgue spaces $L^p(P,X)$ when one varies the exponent $p\in[1,\infty)$.

\begin{prop}\label{comparison of norms}
	Let $p$, $p'\in[1,\infty)$ be such that $p<p'$.
	\begin{enumerate}
		\item\label{enu:i} $L_\sigma^{p'}(P,X)\subseteq L_\sigma^p(P,X)$ and $\Vert Y\Vert_{\sigma,p}\leq\Vert Y\Vert_{\sigma,p'}$ for every $Y\in L_\sigma^{p'}(P,X)$.
		\item\label{enu:ii} If with $r:=p'/(p'-p)$ the distortion function $\sigma$ satisfies $\int_0^1\sigma^r (u)du<\infty$ then even $L^{p'}(P,X)\subseteq L_\sigma^p(P,X)$ and $\Vert Y\Vert_{\sigma,p}\leq\Vert Y\Vert_{p'}$ for every $Y\in L^{p'}(P,X)$.
	\end{enumerate}
\end{prop}

\begin{proof}
	Setting $r:=p'/(p'-p)$ it follows from~\eqref{alternative expression norm}, $1/r+1/(p'/p)=1$ and H{\"o}lder's inequality that for each $X$-valued, strongly measurable $Y$ on $(\Omega,\mathcal{F},P)$
	\begin{align*}
	\Vert Y\Vert_{\sigma,p}^p&=\int_0^1\sigma^\frac{1}{r}(u)\sigma^\frac{1}{p'/p}(u)F_{\Vert Y\Vert}^{-1}(u)^p du\\
	&\leq \left(\int_0^1\sigma(u) du\right)^{1/r}\left(\int_0^1\sigma(u)F_{\Vert Y\Vert}^{-1}(u)^{p'} du\right)^{p/p'}=\Vert Y\Vert_{\sigma,p'}^p,
	\end{align*}
	which proves~\ref{enu:i} while~\ref{enu:ii} follows from~\eqref{alternative expression norm}, $1/r+1/(p'/p)=1$, and H\"older's inequality since
	\begin{align*}
	\Vert Y\Vert_{\sigma,p}^p&=\int_0^1\sigma(u)F_{\Vert Y\Vert}^{-1}(u)^p du\leq \left(\int_0^1\sigma^r(u) du\right)^{1/r}\left(\int_0^1F_{\Vert Y\Vert}^{-1}(u)^{p\frac{p'}{p}} du\right)^{p/p'}\\
	&=\left(\int_0^1\sigma^r(u) du\right)^{1/r}\bigl(\mathbb{E}\,\Vert Y\Vert^{p'}\bigr)^{p/p'}
	\end{align*}
	holds for each $X$-valued, strongly measurable $Y$ on $(\Omega,\mathcal{F},P)$.
\end{proof}

For a Banach space $X$ with (continuous) dual space $X^*$ we write as usual $\langle x^*,x\rangle:=x^*(x)$, $x\in X$, $x^*\in X^*$. The dual norm on $X^*$ will also be denoted by $\Vert\cdot\Vert$. If $Z$ is an $X^*$-valued, Bochner integrable random variable on $(\Omega,\mathcal{F},P)$ such that $\mathbb{E}\Vert Z\Vert=1$ then $\sigma_Z:=F_{\Vert Z\Vert}^{-1}$ is a distortion function. For two $X$-valued, strongly measurable $Y_1$, $Y_2$ on $(\Omega,\mathcal{F},P)$ we write $Y_1\sim Y_2$ if they have the same law, i.e., if $P^{Y_1}=P^{Y_2}$.

\begin{prop}\label{continuity of risk measure}
	Let $X$ be a real Banach space and let $Z$ be an $X^*$-valued, Bochner integrable random variable on $(\Omega,\mathcal{F},P)$ such that $\mathbb{E}\Vert Z\Vert=1$. Then, for every $p\in[1,\infty)$
	\[\rho_Z:L_{\sigma_Z}^p(P,X)\rightarrow\mathbb{R},\ Y\mapsto\sup\{\mathbb{E}\langle Z,Y'\rangle:\ Y\sim Y'\}\]
	is a well-defined subadditive, convex functional. Moreover, for $Y_1$, $Y_2\in L_{\sigma_Z}^p(P,X)$ we have
	\[|\rho_Z(Y_1)-\rho_Z(Y_2)|\leq\Vert Y_1-Y_2\Vert_{\sigma,p}.\]
\end{prop}

\begin{proof}
	It follows from $Y\sim Y'$ that $F_{\Vert Y\Vert}^{-1}=F_{\Vert Y'\Vert}^{-1}$. Hence, $Y'\in L_{\sigma_Z}^p(P,X)$ whenever $Y\in L_{\sigma_Z}^p(P,X)$ by~\eqref{alternative expression norm} in Theorem~\ref{Banach space} . From the strong measurability of $Z$ and $Y\in L_{\sigma_Z}^p(P,X)$ it follows immediately that $\omega\mapsto\langle Z(\omega),Y(\omega)\rangle$ is an $\mathbb{R}$-valued random variable on $(\Omega,\mathcal{F},P)$. The rearrangement inequality, the definition of $\sigma_Z$,~\eqref{alternative expression norm} in Theorem~\ref{Banach space} and Proposition~\ref{comparison of norms} imply that for $Y'\sim Y\in L_{\sigma_Z}^p(P,X)$
	\begin{equation*}
	|\mathbb{E}\langle Z,Y'\rangle|\leq\mathbb{E}(\Vert Z\Vert\,\Vert Y'\Vert)\leq\int_0^1 \sigma_Z(u)F_{\Vert Y\Vert}^{-1}(u)du=\Vert Y\Vert_{\sigma,1}\leq\Vert Y\Vert_{\sigma,p},
	\end{equation*}
	which proves that $\rho_Z$ is well-defined and that
	\begin{equation}\label{useful inequality}
	|\rho_Z(Y)|\leq \Vert Y\Vert_{\sigma,p}.
	\end{equation}
	Obviously, $\rho_Z(\lambda Y)=\lambda\rho_Z(Y)$ for all $\lambda>0$. Moreover, from the definition of $\rho_Z$ and strong measurability it follows immediately that $\rho_Z$ is subadditive. Therefore,
	\[\rho_Z(Y_1)=\rho(Y_2+Y_1-Y_2)\leq \rho_Z(Y_2)+\rho_Z(Y_1-Y_2).\]
	Interchanging the roles of $Y_1$, $Y_2$ in the above inequality gives
	\[|\rho_Z(Y_1)-\rho_Z(Y_2)|\leq\rho_Z(Y_1-Y_2),\]
	which together with~\eqref{useful inequality} proves $|\rho_Z(Y_1)-\rho_Z(Y_2)|\leq\Vert Y_1-Y_2\Vert_{\sigma,p}$.
\end{proof}

In the remainder of this section we will provide a closer look at the Banach spaces $L_\sigma^p(P,X)$.

\begin{prop}\label{bounded sigma}
	Let $X\neq\{0\}$. Then the following are equivalent.
	\begin{enumerate}
		\item\label{enum:1} For all $p\in [1,\infty)$ the spaces $L_\sigma^p(P,X)$ and $L^p(P,X)$ are isomorphic as Banach spaces.
		\item\label{enum:2} There is $p\in[1,\infty)$ such that $L_\sigma^p(P,X)=L^p(P,X)$ as sets.
		\item\label{enum:3} $\sigma$ is bounded.
	\end{enumerate}
\end{prop}

\begin{proof}
	Obviously, \ref{enum:1} implies \ref{enum:2}. By Theorem~\ref{Banach space}, $L_\sigma^p(P,X)$ embeds contractively into $L^p(P,X)$. Thus, if~\ref{enum:2} holds, this embedding is onto so that by Banach's Isomorphism Theorem there is $C>0$ such that
	\begin{equation*}
	\forall\,Y\in L^p(P,X):\,\sup_{U\in\mathscr{U}(0,1)}\int_\Omega \sigma(U(\omega))\Vert Y(\omega)\Vert^p dP(\omega)\leq C\int_\Omega\Vert Y(\omega)\Vert^p dP(\omega),
	\end{equation*}
	where $\mathscr{U}(0,1)$ is defined as before. Choose $f\in L^1(P,\mathbb{R})$ and $x\in X$ with $\Vert x\Vert=1$. Then $Y(\omega):=|f(\omega)|^{1/p}x$ defines an element of $L^p(P,X)$ so that for any $U\in\mathscr{U}(0,1)$ we have
	\begin{align*}
	\left|\int_\Omega \sigma(U(\omega)) f(\omega) dP(\omega)\right|&\leq\int_\Omega \sigma(U(\omega))\Vert Y(\omega)\Vert^p dP(\omega)\\
	&\leq C\int_\Omega\Vert Y(\omega)\Vert^p dP(\omega)\\
	&=C\int_\Omega|f(\omega)| dP(\omega)<\infty.
	\end{align*}
	Since $f\in L^1(P,\mathbb{R})$ was chosen arbitrarily it follows that $\sigma\circ U\in L^\infty (P,\mathbb{R})$ which by $U\in\mathscr{U}(0,1)$ and by the fact that $\sigma$ is nondecreasing implies boundedness of $\sigma$. Thus, \ref{enum:3} follows from \ref{enum:2}.
	
	Finally,~\ref{enum:3} and the fact that $L_\sigma^p(P,X)$ embeds contractively into $L^p(P,X)$ for any $p\in[1,\infty)$ implies~\ref{enum:1} by Theorem~\ref{Banach space}.
\end{proof}

\begin{prop}\label{simple function dense}
	We have the following:
	\begin{enumerate}
		\item\label{enu:p1} For every $p\in[1,\infty)$, $L^\infty(P,X)$ embeds contractively into $L_\sigma^p(P,X)$.
		\item\label{enu:p2} Simple functions are dense in $L_\sigma^p(P,X)$ for every $p\in[1,\infty)$.
	\end{enumerate}
\end{prop}

\begin{proof}
	It follows from the definition of quantile function that $0\leq F_{\Vert Y\Vert}^{-1}\leq \Vert Y\Vert_\infty$ for every $X$-valued, strongly measurable $Y$ on $(\Omega,\mathcal{F},P)$ which implies by~\eqref{alternative expression norm} in Theorem~\ref{Banach space}
	\begin{equation*}
	\Vert Y\Vert_{\sigma,p}^p=\int_0^1 \sigma(u)F_{\Vert Y\Vert}^{-1}(u)^p du\leq \Vert Y\Vert_\infty^p,
	\end{equation*}
	proving~\ref{enu:p1}.
	
	In order to prove~\ref{enu:p2} let $Y\in L_\sigma^p(P,X)$ and fix $\varepsilon\in (0,1)$. We choose $u_\varepsilon\in (0,1)$ such that $\int_{u_\varepsilon}^1 \sigma(u)F_{\Vert Y\Vert}^{-1}(u)^p du<\varepsilon^p$. By the strong measurability of $Y$ there are $N\in\mathcal{F}$ with $P(N)=0$ and a separable, closed subspace $X_1$ of $X$ such that $\1_{N^c}Y$ is $X_1$-valued. Let $\{x_j;\,j\in\mathbb{N}\}$ be a dense subset of $X_1$. Denoting the open ball about $x_j$ with radius $\varepsilon$ in $X$ by $B_\varepsilon(x_j)$ we choose Borel subsets $E_j\subseteq B_\varepsilon(x_j)$ such that $X_1\subseteq\cup_{j\in\mathbb{N}}E_j$ and such that the $E_j$ are pairwise disjoint. Then $((\1_{N^c}Y)^{-1}(E_j))_{j\in\mathbb{N}}$ is a pairwise disjoint sequence in $\mathcal{F}$ such that $P(\cup_{j\in\mathbb{N}}(\1_{N^c}Y)^{-1}(E_j))=1$. Let $n\in\mathbb{N}$ be such that 
	\begin{equation}\label{n for simple function}
	\sum_{j=1}^n P((\1_{N^c}Y)^{-1}(E_j))>u_\varepsilon
	\end{equation}
	and set $E:=\cup_{j=1}^n (\1_{N^c}Y)^{-1}(E_j)$.
	
	Obviously, for $t\geq 0$ we have $\{\1_{E^c}\Vert Y\Vert^p\leq t\}\supseteq \{\Vert Y\Vert^p\leq t\}$ so that $F_{\1_{E^c}\Vert Y\Vert^p}(t)\geq F_{\Vert Y\Vert^p}(t)$. Therefore,
	\begin{equation*}
	\forall\,u\in[0,1]:\,\{t\geq 0;\,F_{\1_{E^c}\Vert Y\Vert^p}(t)\geq u\}\supseteq\{t\geq 0;\,F_{\Vert Y\Vert^p}(t)\geq u\},
	\end{equation*}
	which implies $F_{\1_{E^c}\Vert Y\Vert^p}^{-1}\leq F_{\Vert Y\Vert^p}^{-1}$. Furthermore, 
	\[F_{\1_{E^c}\Vert Y\Vert^p}(0)=P(\1_{E^c}\Vert Y\Vert^p=0)\geq P(E)\]
	so that $F_{\1_{E^c}\Vert Y\Vert^p}^{-1}(u)=0$ for every $u\in[0,P(E)]$, which together with $F_{\Vert Y\Vert^p}^{-1}=(F_{\Vert Y\Vert}^{-1})^p$ and~\eqref{n for simple function} yields for all $u\in[0,1]$
	\begin{equation}\label{estimate for quantiles}
	\,F_{\1_{E^c}\Vert Y\Vert^p}^{-1}(u)\leq\1_{(P(E),1]}(u)F_{\Vert Y\Vert}^{-1}(u)^p\leq\1_{(u_\varepsilon,1]}(u)F_{\Vert Y\Vert}^{-1}(u)^p.
	\end{equation}
	Defining $Y_\varepsilon:=\sum_{j=1}^n\1_{(\1_{N^c}Y)^{-1}(E_j)}\, x_j$ it follows from the definition of $E$ that $\Vert \1_{N^c}Y-Y_\varepsilon\Vert\leq\varepsilon$ on $E$ while $Y_\varepsilon=0$ on $E^c$. For every $U\in\mathscr{U}(0,1)$ we obtain
	\begin{align*}
	\int_\Omega \sigma(U)\Vert \1_{N^c}Y-Y_\varepsilon\Vert^p dP&=\int_\Omega\sigma(U)\1_E\Vert \1_{N^c}Y-Y_\varepsilon\Vert^p dP+\int_\Omega\sigma(U)\1_{E^c}\Vert Y\Vert^p dP\\
	&\leq\varepsilon^p\int_\Omega\sigma(U)dP+\int_0^1\sigma(u)F_{\1_{E^c}\Vert Y\Vert^p}^{-1}(u)du\\
	&\leq\varepsilon^p+\int_{u_\varepsilon}^1 \sigma(u)F_{\Vert Y\Vert^p}^{-1}(u)^p du<2\varepsilon^p,
	\end{align*}
	where we used the rearrangement inequality (see \citet[Theorem 5.25(2)]{McNeil2005}) in the first inequality and~\eqref{estimate for quantiles} in the second one while the last inequality follows form the choice of $u_\varepsilon$. Thus, $\Vert Y-Y_\varepsilon\Vert<\sqrt[p]{2}\varepsilon$, proving~\ref{enu:p2}.
\end{proof}

\begin{thm}\label{Hilbert space}
	For $X\neq\{0\}$ the following are equivalent.
	\begin{enumerate}
		\item\label{enu:H1} $L_\sigma^p(P,X)$ is a Hilbert space.
		\item\label{enu:H2} $X$ is a Hilbert space, $p=2$, and $\sigma=1$ on $(0,1)$. 
	\end{enumerate}
\end{thm}

\begin{proof}
	Obviously, \ref{enu:H2} implies~\ref{enu:H1}.
	
	Let $E_\alpha\in\mathcal{F}$ be chosen with $P(E_\alpha)=\alpha$. For $\alpha\in(0,1)$ and $x\in X$ a straightforward calculation gives for $Y=\1_{E_\alpha}x$ that $F_{\Vert Y\Vert}^{-1}=\Vert x\Vert\1_{(1-\alpha,1)}$. Moreover, for $x_1$, $x_2\in X$ with $\Vert x_1\Vert=\Vert x_2\Vert=1$ and $Y_1:=\1_{E_\alpha}x_1, Y_2:=\1_{E_\alpha^c}x_2$ we have $\Vert Y_1\pm Y_2\|_{\sigma,p}=1$. Thus, by the parallelogram identity and~\eqref{alternative expression norm} we obtain for arbitrary $\alpha\in[0,1]$ 
	\begin{align}\label{integral identity}
	1&=\frac{1}{2}\Vert Y_1- Y_2\|_{\sigma,p}^2+\frac12\Vert Y_1+ Y_2\|_{\sigma,p}^2=\Vert Y_1\|_{\sigma,p}^2+\Vert Y_2\|_{\sigma,p}^2\\
	&=\left(\int_{1-\alpha}^1\sigma(u) du\right)^{2/p}+\left(\int_\alpha^1\sigma(u) du\right)^{2/p}.\nonumber
	\end{align}
	Pick $\alpha\in(0,1)$ so that $0<\int_0^\alpha\sigma(u)du<1$. 
	If $p>2$, then 
	\begin{align*}
	1&=\left(\int_{1-\alpha}^1\sigma(u) du\right)^{2/p}+\left(\int_\alpha^1\sigma(u) du\right)^{2/p}\\
	&> \int_{1-\alpha}^1\sigma(u) du+\int_\alpha^1\sigma(u) du\\
	&\ge \int_{0}^\alpha\sigma(u) du+\int_\alpha^1\sigma(u) du=1,
	\end{align*}
	as $\sigma$ is nondecreasing.
	This is a contradiction and hence $p\le 2$.
	
	Define the function 
	\begin{equation}\label{eq:2}
	f(\alpha):=\left(\int_\alpha^1\sigma(u)du\right)^{2/p}.
	\end{equation}
	Since $\sigma$ is continuous from the left $f$ is differentiable from the left with increasing left derivative, thus $f$ is convex. Further we have $f(\alpha)+f(1-\alpha)=1$ by~\eqref{integral identity} so that $f$ is concave as well. Hence, $f$ is affine, i.e., $f(\alpha)=b+c\cdot\alpha$, and
	we deduce from $f(1)=0, f(0)=1$, and~\eqref{eq:2} the particular form 
	\begin{equation}\label{eq:sigma}
	\sigma(u)=\frac p2(1-u)^{\frac p2-1}
	\end{equation}
	which implies $\int_{1-\alpha}^1\sigma(u)du=\alpha^{p/2}$.
	
	Next consider measurable sets $A$ and $B$ with $A\subseteq B$. The parallelogram law~\eqref{integral identity}, applied to the random variables $Y_1:=\1_A x$ and $Y_2:=\1_B x$, reads 
	\[ \frac12 \big(P(B)-P(A)\big)+\frac12\left(P(B)^{p/2}-P(A)^{p/2}+2^pP(A)^{p/2}\right)^{2/p}=P(A)+P(B),\] i.e.,
	\[ P(B)^{p/2}+(2^p-1)P(A)^{p/2}=\left(3P(A)+P(B)\right)^{p/2}.\]
	We may specify the sets further by $P(B)=4 P(A)$, then the latter equality reduces to
	\[ 4^{p/2}+(2^p-1)=7^{p/2}\]
	so that we are left with solving the equation
	\[2x-1=x^\frac{\log 7}{2\log 2}\] for $x=2^p$.
	
	The convex function $x^\frac{\log 7}{\log 4}$ does not have more than two intersections with the line $2x-1$, and these are $x=1$ and $x=4$, i.e., $p=0$ and $p=2$.
	
	$p=0$ does not qualify, and the distortion function for $p=2$ is $\sigma(\cdot)=1$, by~\eqref{eq:sigma}. This concludes the proof.

\end{proof}

\section{\label{sec:ScalarDual}The dual space in the scalar valued case}

In this section we are going to determine the dual space of $L_\sigma^p:=L_\sigma^p(P):=L_\sigma^p(P,\mathbb{K})$, $\mathbb{K}\in\{\mathbb{R},\mathbb{C}\}$. For $\varphi\in L_\sigma^p(P)^*$ we denote the dual norm of $\varphi$ by $\Vert\varphi\Vert_{\sigma,p}^*$. Some of the results presented in this section are inspired by \citet{Lorentz}.

\begin{defn}
	As usual we denote by $L^0(P)$ the set of $\mathbb{K}$-valued random variables on $(\Omega,\mathcal{F},P)$, where random variables which coincide $P$-almost surely are identified. We define the \textit{K\"othe dual} of $L_\sigma^p(P)$ as
	\[L_\sigma^p(P)^\times:=\{Z\in L^0(P);\,\forall\,Y\in L_\sigma^p(P):\,ZY\in L^1(P)\}.\]
	Since $L^\infty(P)\subseteq L_\sigma^p(P)$ for all $p\in[1,\infty)$ it follows from taking $Y=\1_{\{Z\neq 0\}}\frac{\overline{Z}}{|Z|}$ that $Z\in L^1(P)$ whenever $Z\in L_\sigma^p(P)^\times$.
\end{defn}

\begin{prop}\label{Koethe dual}
	For every $Z\in L_\sigma^p(P)^\times$
	\[\sup\{|\mathbb{E}(ZY)|;\,\Vert Y\Vert_{\sigma,p}\leq 1\}<\infty.\]
	Moreover,
	\[\varphi_Z:L_\sigma^p(P)\rightarrow\mathbb{K},\ \varphi_Z(Y)=\mathbb{E}(ZY)\]
	belongs to $L_\sigma^p(P)^*$ and
	\[\Phi: L_\sigma^p(P)^\times\rightarrow L_\sigma^p(P)^*,\ Z\mapsto\varphi_Z\]
	is a linear isomorphism with
	\begin{equation}\label{dual norm Koethe}
	\forall\,Z\in L_\sigma^p(P)^\times:\,\Vert\Phi(Z)\Vert_{\sigma,p}^*=\sup\{|\mathbb{E}(ZY)|;\,\Vert Y\Vert_{\sigma,p}\leq 1\}.
	\end{equation} 
\end{prop}

\begin{proof}
	Obviously, $\varphi_Z$ is a well-defined, linear functional on $L_\sigma^p(P)$ for every $Z\in L_\sigma^p(P)^\times$. The assumption
	\[\infty=\sup\{|\mathbb{E}(ZY)|;\,\Vert Y\Vert_{\sigma,p}\leq 1\}\]
	implies the existence of a sequence $(Y_k)_{k\in\mathbb{N}}$ in the unit ball of $L_\sigma^p(P)$ such that
	\[\forall\,k\in\mathbb{N}:\,k^2\leq |\mathbb{E}(ZY_k)|\leq \mathbb{E}(|ZY_k|).\]
	Because $\tilde{Y}_k:=\1_{\{ZY_k\neq 0\}}\frac{\overline{ZY_k}}{|Z\tilde{Y}_k|}Y_k$ belongs to the unit ball of $L_\sigma^p(P), k\in \mathbb{N}$ the completeness of $L_\sigma^p(P)$ implies that $(\sum_{k=1}^n\frac{1}{k^2}\tilde{Y}_k)_{n\in\mathbb{N}}$ converges in $L_\sigma^p(P)$ to some $Y$. As $Z\in L_\sigma^p(P)^\times$ it follows that $ZY\in L^1(P)$. 
	
	But on the other hand, since $L_\sigma^p(P)$ embeds contractively into $L^p(P)$ by Theorem~\ref{Banach space} it follows that some subsequence $(\sum_{k=1}^{n_l}\frac{1}{k^2}\tilde{Y}_k)_{l\in\mathbb{N}}$ also converges $P$-almost surely to $Y$. Therefore, $P$-almost surely we have 
	\begin{equation}\label{almost sure}
	ZY=Z\lim_{l\rightarrow\infty}(\sum_{k=1}^{n_l}\frac{1}{k^2}\tilde{Y}_k)=\lim_{l\rightarrow\infty}\sum_{k=1}^{n_l}\frac{1}{k^2}Z\tilde{Y}_k=\lim_{l\rightarrow\infty}\sum_{k=1}^{n_l}\frac{1}{k^2}|ZY_k|
	\end{equation}
	and by an application of the Monotone Convergence Theorem we conclude
	\begin{equation*}
	\mathbb{E}(ZY)=\mathbb{E}(\lim_{l\rightarrow\infty}\sum_{k=1}^{n_l}\frac{1}{k^2}|ZY_k|)=\lim_{l\rightarrow\infty}\sum_{k=1}^{n_l}\frac{1}{k^2}\mathbb{E}(|ZY_k|)\geq\lim_{l\rightarrow\infty}\sum_{k=1}^{n_l} 1
	\end{equation*}
	which contradicts $ZY\in L^1(P)$. Hence,
	\[\infty>\sup\{|\mathbb{E}(ZY)|;\,\Vert Y\Vert_{\sigma,p}\leq 1\}\]
	so that $\varphi_Z\in L_\sigma^p(P)^*$ with
	\begin{equation}\label{dual norm Koethe 2}
	\Vert\varphi_Z\Vert_{\sigma,p}^*=\sup\{|\mathbb{E}(ZY)|;\,\Vert Y\Vert_{\sigma,p}\leq 1\}.
	\end{equation}
	This implies that $\Phi$ is a well-defined linear mapping which satisfies~\eqref{dual norm Koethe}. In order to show that $\Phi$ is injective choose $Z\in L_\sigma^p(P)^\times$ with $\Phi(Z)=0$. We set $Y:=\1_{\{Z\neq 0\}}\frac{\overline{Z}}{|Z|}$. Since simple functions belong to $L_\sigma^p(P)$ it follows easily that $Y\in L_\sigma^p(P)$. It follows
	\[0=\Phi(Z)(Y)=\mathbb{E}(|Z|),\]
	so that $Z=0$.
	
	In order to prove surjectivity of $\Phi$ let $\varphi\in L_\sigma^p(P)$. For $E\in\mathcal{F}$ and $Y=\1_E$ we have $F_{|Y|}^{-1}=\1_{(1-P(E),1]}$ so that by~\eqref{alternative expression norm}
	\[|\varphi(\1_E)|\leq\Vert\varphi_{\sigma,p}^*\Vert\,\Vert\1_E\Vert_{\sigma,p}=\Vert\varphi\Vert_{\sigma,p}^*(\int_{1-P(E)}^1\sigma(u)du)^{1/p}.\]
	Using this inequality, it is straightforward to show that
	\[\mu:\mathcal{F}\rightarrow\mathbb{K},\ \mu(E):=\varphi(\1_E)\]
	is a complex measure which is $P$-continuous, i.e., $\mu(E)=0$ whenever $P(E)=0$. An application of the Radon-Nikod\'ym Theorem yields some $Z\in L^1(P)$ such that $\mu(E)=\int_\Omega\1_E ZdP=\mathbb{E}(Z\1_E)$ for all $E\in\mathcal{F}$. For simple functions $Y$ it follows $\varphi(Y)=\mathbb{E}(ZY)$. As soon as we have shown that $Z\in L_\sigma^p(P)^\times$ it follows from the above and Theorem~\ref{simple function dense} that $\varphi=\Phi(Z)$.
	
	In order to show $Z\in L_\sigma^p(P)^\times$ we first observe that $\alpha Y\in L_\sigma^p(P)$ and $\Vert\alpha Y\Vert_{\sigma,p}\leq\Vert\alpha\Vert_\infty\,\Vert Y\Vert_{\sigma,p}$ for every $Y\in L_\sigma^p(P)$ and each $\alpha\in L^\infty(P)$. Therefore, by setting $E_n:=\{|Z|\leq n\}, n\in\mathbb{N},$ we have $\Vert \1_{E_n}Y\Vert_{\sigma,p}\leq\Vert Y\Vert_{\sigma,p}$ for each $Y\in L_\sigma^p(P)$ which implies $\varphi_n\in L_\sigma^p(P)^*$ and $\Vert\varphi_n\Vert_{\sigma,p}^*\leq \Vert\varphi\Vert_{\sigma,p}^*$ where $\varphi_n(Y):=\varphi(\1_{E_n}Y)$. For simple functions $Y$ we have $\varphi_n(Y)=\mathbb{E}(Z\1_{E_n}Y)$. Additionally, by H\"older's inequality and Theorem~\ref{Banach space} we obtain for arbitary $Y\in L_\sigma^p(P)$
	\[\mathbb{E}(|Z\1_{E_n}Y|)\leq n\mathbb{E}(|Y|)\leq n\Vert Y\Vert_{\sigma,p},\]
	so that $Z\1_{E_n}\in L_\sigma^p(P)^\times$. Because simple functions are dense in $L_\sigma^p(P)$ by Theorem~\ref{simple function dense} we conclude from the above $\Phi(Z\1_{E_n})=\varphi_n$. Finally, since
	\begin{align*}
	\mathbb{E}(|Z\1_{E_n}Y|)&=|\varphi_n(\1_{\{ZY\neq 0\}}\frac{\overline{Z Y}}{|ZY|}Y)|\leq\Vert\varphi_n\Vert_{\sigma,p}^*\Vert\1_{\{ZY\neq 0\}}\frac{\overline{Z Y}}{|ZY|}Y\Vert_{\sigma,p}\\
	&\leq\Vert\varphi\Vert_{\sigma,p}\Vert Y\Vert_{\sigma,p}
	\end{align*}
	it follows with the aid of the Monotone Convergence Theorem that
	\[\mathbb{E}(|ZY|)=\lim_{n\rightarrow\infty}\mathbb{E}(|Z\1_{E_n}Y|)\leq \Vert\varphi\Vert_{\sigma,p}\Vert Y\Vert_{\sigma,p}\]
	for each $Y\in L_\sigma^p(P)$ so that $Z\in L_\sigma^p(P)^\times$.
\end{proof}

\begin{rem}\label{Koethe dual enthaelt betraege}
	With the aid of the fact that for $\alpha\in L^\infty(P)$ the linear mapping $Y\mapsto \alpha Y$ is well defined and continuous from $L_\sigma^p(P)$ into itself it is straightforward to see that $|Z|\in L_\sigma^p(P)^\times$ whenever $Z\in L_\sigma^p(P)^\times$ and that in this case $\Vert\varphi_Z\Vert_{\sigma,p}^*=\Vert\varphi_{|Z|}\Vert_{\sigma,p}^*$.
\end{rem}

We next aim at giving a representation of $L_\sigma^p(P)^\times$ and thus of the dual space of $L_\sigma^p(P)$. For this purpose we introduce the following notion.

\begin{defn}\label{S}
	For a distortion function $\sigma$ we define
	\[S_\sigma:=S:[0,1]\rightarrow\mathbb{R},\ S_\sigma(\alpha)=\int_\alpha^1 \sigma(u)du.\]
\end{defn}

\begin{rem}\label{properties of S}
	Obviously, $S$ is a continuous, nonincreasing function with $S(0)=1, S(1)=0$. If we set $u_0:=\inf\{u>0;\sigma(u)>0\}$ we have $u_0<1$, $S_{|[0,u_0]}=1$ and $S_{|[u_0,1]}$ is an increasing bijection from $[u_0,1]$ to $[0,1]$. By abuse of notation we denote the inverse of $S_{|[u_0,1]}$ by $S^{-1}$.
	
	For $\alpha_1$, $\alpha_2\in[0,1]$, $\alpha_1<\alpha_2$ and $\lambda\in(0,1)$ it follows from the fact that $\sigma$ is nondecreasing that
	\begin{align*}
	S(\lambda\alpha_1+(1-\lambda)\alpha_2)-S(\alpha_1)=&-\int_{\alpha_1}^{\lambda\alpha_1+(1-\lambda)\alpha_2}\sigma(u)du\\
	\geq& -\sigma(\lambda\alpha_1+(1-\lambda)\alpha_2)\,(1-\lambda)(\alpha_2-\alpha_1)
	\end{align*}
	and
	\begin{align*}
	S(\alpha_2)-S(\lambda\alpha_1+(1-\lambda)\alpha_2)=&-\int_{\lambda\alpha_1+(1-\lambda)\alpha_2}^{\alpha_2}\sigma(u)du\\
	\leq& -\sigma(\lambda\alpha_1+(1-\lambda)\alpha_2)\,\lambda(\alpha_2-\alpha_1)
	\end{align*}
	so that
	\begin{align*}
	\frac{S(\lambda\alpha_1+(1-\lambda)\alpha_2)-S(\alpha_1)}{(1-\lambda)(\alpha_2-\alpha_1)}&\geq-\sigma(\lambda\alpha_1+(1-\lambda)\alpha_2)\\
	&\geq\frac{S(\alpha_2)-S(\lambda\alpha_1+(1-\lambda)\alpha_2)}{\lambda(\alpha_2-\alpha_1)},
	\end{align*}
	which implies $S(\lambda\alpha_1+(1-\lambda)\alpha_2)\geq\lambda S(\alpha_1)+(1-\lambda)S(\alpha_2)$, i.e., $S$ is concave. In particular, $S$ is differentiable from the left and from the right on $(0,1]$, (on $[0,1)$, resp.)
	and since $\sigma$ is continuous from the left it is straightforward to show that for the left derivative we have $S_l'(\alpha)=-\sigma(u), u\in (0,1]$.
\end{rem}

Recall that for a non-negative random variable $Z$ the \textit{average value-at-risk} of level $\alpha\in[0,1)$ is defined as $\AVaR_\alpha(Z)=\frac{1}{1-\alpha}\int_\alpha^1 F_Z^{-1}(u)du$.

\begin{defn}
	For a distortion function $\sigma$, $Z\in L^0(P)$, and $\alpha\in[0,1)$ we define
	\begin{equation}\label{average value at risk norm}
	|Z|_{\sigma,\infty}^*:=\sup_{\alpha\in[0,1)}\frac{\AVaR_\alpha(|Z|)}{\frac{1}{1-\alpha}S_\sigma(\alpha)}\ (=\sup_{\alpha\in[0,1)}\frac{\int_\alpha^1 F_{|Z|}^{-1}(u)du}{\int_\alpha^1\sigma(u)du}).
	\end{equation}
	Moreover, we say that $Z^{\prime}\in L^0(P)$ \textit{$\sigma$\nobreakdash-dominates} $Z$ (in symbols $Z^{\prime}\:{}_{\sigma}\negthickspace\succcurlyeq Z$) if
	there is a uniform random variable $U\in\mathscr{U}(0,1)$ such that 
	\begin{equation}\label{eq:Order}
	\AVaR_{\alpha}\bigl(\sigma(U)|Z^{\prime}|\bigr)\ge\AVaR_{\alpha}(|Z|)\text{ for all }\alpha<1.
	\end{equation}
	Further we define, for $p\in (1,\infty)$,
	\begin{equation}\label{eq:20}
	|Z|_{\sigma,q}^{*}:=\inf\left\{ \left\Vert Z^{\prime}\right\Vert _{\sigma,q}:Z^{\prime}{}_{\sigma}\negthickspace\succcurlyeq Z \right\} 
	\end{equation}
	where $q\in (1,\infty)$ is the conjugate exponent to $p$, i.e., $1/p +1/q=1$ and where as usual $\inf\emptyset:=\infty$.
	
	Finally, for $p\in [1,\infty)$ with conjugate exponent $q$, i.e. $1/p+1/q=1$, we set $L_{\sigma,q}^*(P):=\{Z\in L^0(P);\,|Z|_{\sigma,q}^*<\infty\}$ (and we identify random variables which coincide $P$-almost everywhere).
\end{defn}

From the definition of quantile functions it follows for $Z_1$, $Z_2\in L^0(P)$ with $|Z_1|\leq |Z_2|$ that $F_{|Z_1|}^{-1}\leq F_{|Z_2|}^{-1}$ which implies $|Z_1|_{\sigma,q}^*\leq |Z_2|_{\sigma,q}^*$. Since also $F_{|\alpha Z_1|}^{-1}=|\alpha|F_{|Z_1|}^{-1}$ for $\alpha\in\mathbb{K}$ it follows also $|\alpha Z_1|_{\sigma,q}^*=|\alpha|\,|Z_1|_{\sigma,q}^*$. Since $\AVaR_\alpha$ is subadditive (cf.\ \citet{PflugRomisch2007}) it follows easily that $L_{\sigma,q}^*(P)$ is a subspace of $L^0(P)$.

\begin{rem}[Stochastic dominance of second order]
	
	The definition of $| \cdot|_{\sigma,\infty}^{*}$ reflects
	the duality of risk functionals. Indeed, the supremum~\eqref{average value at risk norm}
	can be restated as 
	\[
	| Z|_{\sigma,\infty}^{*}=\inf\left\{ \eta\geq 0:\AVaR_{\alpha}(| Z| )\leq\frac{\eta}{1-\alpha}\cdot\int_{\alpha}^{1}\sigma(u)\mathrm{d}u\text{ for all }\alpha<1\right\}.
	\]
	By the rearrangement inequality (cf.\ \citet[Theorem 5.25(2)]{McNeil2005}) this equivalent formulation involves the statement 
	\begin{equation}\label{eq:12}
	\AVaR_{\alpha}(|Z|)\le\AVaR_{\alpha}\bigl(\eta\,\sigma(U)\bigr),
	\end{equation}
	where $U\in\mathscr{U}(0,1)$. Choosing $U$ to be coupled in a comonotone way with $|Z|$ it follows
	\[| Z|_{\sigma,\infty}^{*}=\inf\left\{ \eta\geq 0, U\in\mathscr{U}(0,1):\AVaR_{\alpha}(| Z| )\leq\AVaR_{\alpha}\bigl(\eta\,\sigma(U)\bigr)\text{ for all }\alpha<1\right\} \]
	Following \citet{RuszOgryczak}, \eqref{eq:12} is equivalent to saying
	that $|Z|$ is dominated by $\left\Vert Z\right\Vert _{\sigma}^{*}\cdot\sigma(U)$
	in second stochastic order.\footnote{Cf.\ \citet{Dentcheva2003,DentchevaRusz2004,Ruszczynski2006a} for
		stochastic dominance of second order.} 
\end{rem}

\begin{rem}~
	\begin{enumerate}
		\item By the choice $\alpha=0$ in~\eqref{average value at risk norm} it follows that 
		\begin{equation}
		| Z|_{\sigma,\infty}^{*}\ge\AVaR_{0}(|Z|)=\int_{0}^{1}F_{|Z|}^{-1}(u)\mathrm{d}u=\mathbb{E}|Z| =\left\Vert Z\right\Vert _{1},\label{eq:12-2}
		\end{equation}
		so that $L_{\sigma,\infty}^*(P)\subseteq L^1(P)$.
		\item Since for $Z$, $Z^\prime\in L^0(P)$ with $Z^{\prime}{}_{\sigma}\negthickspace\succcurlyeq Z$ we have for $p\in[1,\infty)$ with Proposition~\ref{comparison of norms}
		\[\Vert Z\Vert_1=\int_0^1F_{|Z|}^{-1}(u)du\leq\int_0^1\sigma(u)F_{|Z^\prime|}^{-1}(u)du=\Vert Z^\prime\Vert_{\sigma,1}\leq\Vert Z^\prime\Vert_{\sigma,p}\]
		it also follows that $L_{\sigma,q}^*(P)\subseteq L^1(P)$.
	\end{enumerate}
\end{rem}

\begin{prop}\label{description avar norm}
	For $Z\in L^0(P)$ we have
	\begin{align*}
	| Z|_{\sigma,\infty}^*&=\inf\{\eta\geq 0;\,\forall\,F:[0,1)\rightarrow[0,\infty)\text{ nondecreasing:}\\
	& \int_0^1 F_{|Z|}^{-1}(u)F(u)du\leq \eta\int_0^1 \sigma(u) F(u)du\},
	\end{align*}
	and for $p\in (1,\infty)$
	\begin{align*}
	| Z|_{\sigma,q}^*&=\inf\{\Vert Z^{\prime}\Vert_{\sigma,q};\,Z^{\prime}\in L^0(P)\text{ such that }\forall\,F:[0,1)\rightarrow[0,\infty)\text{ nondecreasing:}\\
	& \int_0^1 F_{|Z|}^{-1}(u)F(u)du\leq \int_0^1 \sigma(u) F_{|Z^\prime|}^{-1}(u)F(u)du\},
	\end{align*}
	where as usual $1/p+1/q=1$.
\end{prop}

\begin{proof}
	Since $\1_{[\alpha,1]}$ is a nondecreasing, non-negative function for every $\alpha\in[0,1)$ it follows
	\begin{align*}
	\{\eta\geq 0&;\,\forall\,F:[0,1)\rightarrow[0,\infty)\text{ nondecreasing:} \int_0^1 F_{|Z|}^{-1}(u)F(u)du\leq \eta\int_0^1 \sigma(u) F(u)du\}\\
	&\subseteq \{\eta\geq 0;\,\forall\,\alpha\in[0,1):\, \int_\alpha^1 F_{|Z|}^{-1}(u)du\leq \eta\int_\alpha^1 \sigma(u)du\}.
	\end{align*}
	On the other hand, if for some $\eta\geq 0$ we have
	\[\forall\,\alpha\in[0,1):\,\int_\alpha^1 F_{|Z|}^{-1}(u)du\leq \eta\int_\alpha^1 \sigma(u)du\]
	it follows
	for all $\gamma_1,\ldots,\gamma_n\in[0,\infty)$ and every choice of $\alpha_1<\ldots<\alpha_n\in[0,1)$ that
	\[\int_0^1 F_{|Z|}^{-1}(u)\sum_{j=1}^n\gamma_j\1_{[\alpha_j,1]}(u)du\leq \eta\int_0^1 \sigma(u)\sum_{j=1}^n\gamma_j\1_{[\alpha_j,1]}(u)du.\]
	Since every non-negative, nondecreasing function $F:[0,1)\rightarrow[0,\infty)$ is the pointwise limit of a nondecreasing sequence of such step functions $\sum_{j=1}^n\gamma_j\1_{[\alpha_j,1]}$ it follows from the Monotone Convergence Theorem that
	\[\int_0^1 F_{|Z|}^{-1}(u)F(u)du\leq \eta\int_0^1 \sigma(u) F(u)du\]
	for all such $F$. Hence it also holds
	\begin{align*}
	&\{\eta\geq 0;\,\forall\,F:[0,1)\rightarrow[0,\infty)\text{ nondecreasing:} \int_0^1 F_{|Z|}^{-1}(u)F(u)du\leq \eta\int_0^1 \sigma(u) F(u)du\}\\
	&\supseteq \{\eta\geq 0;\,\forall\,\alpha\in[0,1):\, \int_\alpha^1 F_{|Z|}^{-1}(u)du\leq \eta\int_\alpha^1 \sigma(u)du\}
	\end{align*}
	which proves the first claim. The rest of the proposition is proved mutatis muntandis.
\end{proof}

\begin{prop}\label{one part of duality}
	For a distortion function $\sigma$ and $p\in [1,\infty)$ we have $L_{\sigma,q}^*(P)\subseteq L_\sigma^p(P)^\times$ and
	\begin{equation}\label{isometry part 1}
	\sup\{|\mathbb{E}(ZY)|;\,Y\in L_\sigma^p(P),\Vert Y\Vert_{\sigma,p}\leq 1\}\leq |Z|_{\sigma,q}^*
	\end{equation}
	for every $Z\in L_{\sigma,q}^*(P)$, where $q$ is the conjugate exponent to $p$.
\end{prop}

\begin{proof}
	Let $p=1$. For $Z\in L_{\sigma,\infty}^*(P)$ it follows for arbitrary $Y\in L_\sigma^1(P)$ from the rearrangement inequality combined with Proposition~\ref{description avar norm}
	\begin{align*}
	\mathbb{E}(|ZY|)&\leq\int_0^1 F_{|Z|}^{-1}(u) F_{|Y|}^{-1}(u)du\leq | Z|_{\sigma,\infty}^*\int_0^1\sigma(u)F_{|Y|}^{-1}(u) du\\
	&=| Z|_{\sigma,\infty}^* \Vert Y\Vert_{\sigma,1}.
	\end{align*}
	Hence, $Z\in L_\sigma^1(P)^\times$ and the above inequality also implies that $| Z|_{\sigma,\infty}^*$ is an upper bound for $\sup\{|\mathbb{E}(ZY)|;\,Y\in L_\sigma^1(P), \Vert Y\Vert_{\sigma,1}\leq 1\}.$
	
	Next let $p\in (1,\infty)$ and $q$ the corresponding conjugate exponent. For $Z\in L_{\sigma,q}^*(P)$ let $Z^\prime\in L_\sigma^q(P)$ with $Z^{\prime}{}_{\sigma}\negthickspace\succcurlyeq Z$. For arbitrary $Y\in L_\sigma^1(P)$ it follows from the rearrangement inequality combined with Proposition~\ref{description avar norm} and H\"older's inequality
	\begin{align*}
	\mathbb{E}(|ZY|)&\leq\int_0^1 F_{|Z|}^{-1}(u) F_{|Y|}^{-1}(u)du\leq \int_0^1\sigma(u)F_{|Z^\prime|}^{-1}(u)F_{|Y|}^{-1}(u) du\\
	&\leq\Vert Z^\prime\Vert_{\sigma,q} \Vert Y\Vert_{\sigma,p}.
	\end{align*}
	Thus, $Z\in L_\sigma^p(P)^\times$ and because $Z^\prime\in L_\sigma^q(P)$ with $Z^{\prime}{}_{\sigma}\negthickspace\succcurlyeq Z$ was chosen arbitrarily, it follows that $\sup\{|\mathbb{E}(ZY)|;\,Y\in L_\sigma^1(P), \Vert Y\Vert_{\sigma,p}\leq 1\}$ is bounded by $|Z|_{\sigma,q}^*$.
\end{proof}

In order to show that in fact $L_{\sigma,q}^*(P)=L_\sigma^p(P)^\times$ holds as well as equality in inequality~\eqref{isometry part 1} we have to distinguish the cases $p=1$ and $p\in (1,\infty)$. We begin with the case $p=1$.

\begin{prop}\label{avar equals expectation}
	For a $\mathbb{K}$-valued random variable $Z$ on $(\Omega,\mathcal{F},P)$ and $\alpha\in[0,1)$ there is $E_\alpha\in\mathcal{F}$ such that $P(E_\alpha)=1-\alpha$ and
	\[\AVaR_\alpha(|Z|)=\frac{1}{1-\alpha}\mathbb{E}(|Z|\1_{E_\alpha}).\]
\end{prop}

\begin{proof}
	Let $E\in\mathcal{F}$ with
	\[\{|Z|>F_{|Z|}^{-1}(\alpha)\}\subseteq E\subseteq\{|Z|\geq F_{|Z|}(\alpha)\}\]
	be arbitrary. Denoting the positive part of an $\mathbb{R}$-valued function $f$ as usual by $f_+$ it follows
	\begin{align}\label{avar equality}
	\AVaR_\alpha(|Z|)&=\frac{1}{1-\alpha}\int_\alpha^1 F_{|Z|}^{-1}(u)du\nonumber\\
	&=F_{|Z|}^{-1}(\alpha)+\frac{1}{1-\alpha}\int_0^1(F_{|Z|}^{-1}(u)-F_{|Z|}^{-1}(\alpha))_+du\nonumber\\
	&=F_{|Z|}^{-1}(\alpha)+\frac{1}{1-\alpha}\mathbb{E}\bigl((|Z|-F_{|Z|}^{-1}(\alpha))\1_E\bigr)\\
	&=F_{|Z|}^{-1}(\alpha)+\frac{1}{1-\alpha}\mathbb{E}(|Z|\1_E)-\frac{1}{1-\alpha}F_{|Z|}^{-1}(\alpha)P(E)\nonumber\\
	&=\frac{1}{1-\alpha}\mathbb{E}(|Z|\1_E)+(1-\frac{1}{1-\alpha}P(E))F_{|Z|}^{-1}(\alpha).\nonumber
	\end{align}
	From the definition of $F_{|Z|}^{-1}$ it follows immediately that
	\[P(|Z|<F_{|Z|}^{-1}(\alpha))\leq\alpha\leq P(|Z|\leq F_{|Z|}^{-1}(\alpha))\]
	so that
	\begin{equation*}
	P(|Z|>F_{|Z|}^{-1}(\alpha))\leq 1-\alpha\text{ and }P(|Z|\geq F_{|Z|}^{-1}(\alpha))\geq 1-\alpha.
	\end{equation*}
	Let $U$ be a $[0,1]$-valued, uniformly distributed random variable on $(\Omega,\mathcal{F},P)$. We define
	\[E_\beta:=\{|Z|>F_{|Z|}^{-1}(\alpha)\}\cup\Bigl(\{|Z|=F_{|Z|}^{-1}(\alpha)\}\cap\{U\in [0,\beta]\}\Bigr)\]
	for $\beta\in[0,1]$ and set
	\[f:[0,1]\rightarrow[0,1], f(\beta):=P(E_\beta).\]
	From the properties of a probability measure it follows easily that $f$ is continuous as well as
	\[f(0)=P(|Z|>F_{|Z|}^{-1}(\alpha))\leq 1-\alpha\text{ and }f(1)=P(|Z|\geq F_{|Z|}^{-1}(\alpha))\geq 1-\alpha.\]
	Hence, there is $\beta_0\in [0,1]$ such that for $E_{\beta_0}$ we have $P(E_{\beta_0})=1-\alpha$ and it follows from~\eqref{avar equality} that $E_{\beta_0}$ has the desired property.
\end{proof}

For the case $p=1$ we can now give the desired intrinsic description of $L_\sigma^1(P)^\times$.

\begin{prop}\label{representation Koethe dual p=1}
	For a distortion function $\sigma$ it holds $L_\sigma^1(P)^\times=L_{\sigma,\infty}^*(P)$
	and for every $Z\in L_\sigma^1(P)^\times$ we have
	\[| Z|_{\sigma,\infty}^*=\sup\{|\mathbb{E}(ZY)|;\,Y\in L_\sigma^1(P), \Vert Y\Vert_{\sigma,1}\leq 1\}.\]
\end{prop}

\begin{proof}
	Let $Z\in L_\sigma^1(P)^\times$. By Proposition~\ref{avar equals expectation}, for any $\alpha\in[0,1)$ there is $E_\alpha\in\mathcal{F}$ such that $\AVaR_\alpha(|Z|)=\frac{1}{1-\alpha}\mathbb{E}(|Z|\1_{E_\alpha})$ and $P(E_\alpha)=1-\alpha$. Employing the notation from Proposition~\ref{Koethe dual} we obtain
	\begin{align*}
	\AVaR_\alpha(|Z|)&=|\varphi_Z(\frac{1}{1-\alpha}\1_{\{Z\neq 0\}}\frac{\overline{Z}}{|Z|}\1_{E_\alpha})|\leq\Vert\varphi_Z\Vert_{\sigma,1}^*\frac{1}{1-\alpha}\Vert \1_{E_\alpha}\Vert_{\sigma,1}\\
	&=\Vert\varphi_Z\Vert_{\sigma,1}^*\frac{1}{1-\alpha}\int_0^1\sigma(u)F_{\1_{E_\alpha}}^{-1}(u)du\\
	&=\Vert\varphi_Z\Vert_{\sigma,1}^*\frac{1}{1-\alpha}\int_{1-P(E_\alpha)}^1\sigma(u)du=\Vert\varphi_Z\Vert_{\sigma,1}^*\frac{1}{1-\alpha}S_\sigma(\alpha),
	\end{align*}
	so that $| Z|_{\sigma,\infty}^*$ is finite and bounded above by \[\Vert\varphi_Z\Vert_{\sigma,1}^*=\sup\{|\mathbb{E}(ZY)|;\,Y\in L_\sigma^1(P), \Vert Y\Vert_{\sigma,1}\leq 1\}.\]
	Proposition~\ref{one part of duality} now yields the rest of the claim.
\end{proof}

Combining Propositions~\ref{Koethe dual} and~\ref{representation Koethe dual p=1} we immediately derive the next result.

\begin{thm}\label{dual representation scalar case p=1}
	Let $\sigma$ be a distortion function. Then $|\cdot|_{\sigma,\infty}^*$ is a norm on $L_{\sigma,\infty}^*(P)$ turning it into a Banach space. Moreover,
	\[\Phi:(L_{\sigma,\infty}^*(P),|\cdot|_{\sigma,\infty}^*)\rightarrow(L_\sigma^1(P)^*,\Vert\cdot\Vert_{\sigma,1}^*),\ Z\mapsto(Y\mapsto\Phi(Z)(Y):=\mathbb{E}(ZY))\]
	is an isometric isomorphism.
\end{thm}

In order to derive an analogous representation for the case $p\in (1,\infty)$ we need an equivalent result to~\ref{representation Koethe dual p=1} for this case. This requires some preparation. We begin by recalling a notion from \citet{Lorentz}.

\begin{defn}\label{S-concave}
	Let $\sigma$ be a distortion function. A function $H:[0,1]\rightarrow\mathbb{R}$ is called \textit{$S_\sigma$-concave} if whenever $y,b\in\mathbb{R}$ are such that $yS_\sigma(\alpha_1)+b=H(\alpha_1)$ and $yS_\sigma(\alpha_2)+b=H(\alpha_2)$ for some $\alpha_1<\alpha_2\in [0,1]$ then $H(\alpha)\geq yS_\sigma(\alpha)+b$ for each $\alpha\in[\alpha_1,\alpha_2]$.
\end{defn}

The next proposition is essentially contained in \citet[Proof of Theorem 3.6.2]{Lorentz}. Nevertheless, we include its proof for the reader's convenience.

\begin{prop}\label{concavity of infimum}
	Let $\sigma$ be a distortion function and let $u_0:=\inf\{u>0;\,\sigma(u)>0\}$. Moreover, let $\mathscr{H}$ be a set of $S_\sigma$-concave functions such that $H_{|[0,u_0]}$ is constant for every $H\in\mathscr{H}$. Assume that
	\[\forall\,\alpha\in[0,1]:\, F(\alpha):=\inf\{H(\alpha); H\in\mathscr{H}\}>-\infty.\]
	Then, $F$ is $S_\sigma$-concave.
\end{prop}

\begin{proof}
	Let $y,b\in\mathbb{R}$ and let $\alpha_1<\alpha_2\in[0,1]$ such that $yS_\sigma(\alpha_j)+b=F(\alpha_j), j=1,2$. Let $H\in\mathscr{H}$ be arbitrary. Since $H_{|[0,u_0]}$ is constant there are $\bar{y},\bar{b}\in\mathbb{R}$ such that $\bar{y} S_\sigma(\alpha_j)+\bar{b}=H(\alpha_j), j=1,2$.
	
	In case of $y-\bar{y}\geq 0$ it follows that $\bar{y}S_\sigma+\bar{b}-(yS_\sigma+b)$ is nonincreasing while $\bar{y}S_\sigma+\bar{b}-(yS_\sigma+b)$ is nondecreasing in case of $y-\bar{y}\leq 0$. Therefore, $S_\sigma$-concavity of $H$ together with $yS_\sigma(\alpha_j)+b\leq\bar{y}S_\sigma(\alpha_j)+\bar{b}, j=1,2$ implies
	\[\forall\,\alpha\in [\alpha_1,\alpha_2]: yS_\sigma(\alpha)+b\leq\bar{y}S_\sigma(\alpha)+\bar{b}\leq H(\alpha).\]
	Since $H\in\mathscr{H}$ was arbitrary, we conclude that $F\geq yS_\sigma+b$ on $[\alpha_1,\alpha_2]$.
\end{proof}

\begin{prop}\label{properties S and S-concave}
	Let $\sigma$ be a distortion function, $u_0:=\inf\{u>0;\,\sigma(u)>0\}$, and let $y_1,y_2,b_1,b_2\in\mathbb{R}$. Moreover, let $H$ be $S_\sigma$-concave, continuous from the right in $u_0$ such that $H_{|[0,u_0]}$ is constant.
	
	If for $y,b\in\mathbb{R}$ and $\alpha_1,\alpha_2\in[0,1]$ with $\alpha_1<\alpha_2$ and $u_0<\alpha_2$ we have $H(\alpha_j)=y S_\sigma(\alpha_j)+b, j=1,2,$ then it follows that $yS_\sigma+b\geq H$ on $[0,1]\backslash (\alpha_1,\alpha_2)$.
	
\end{prop}

\begin{proof}
	It is straightforward to show that if for $\alpha, \beta\in[0,1], \alpha<\beta,$ it holds $y_1S_\sigma(\alpha)+b_1=y_2S_\sigma(\alpha)+b_2$ and $y_1S_\sigma(\beta)+b_1>y_2S_\sigma(\beta)+b_2$
	then $\beta>u_0$ and 
	\begin{equation}\label{inequality 1}
	\forall\,\gamma\in (\max\{\alpha,u_0\},\beta]: y_1S_\sigma(\gamma)+b_1>y_2S_\sigma(\gamma)+b_2
	\end{equation}
	while for $\alpha,\beta\in[u_0,1], \alpha<\beta,$ the conditions $y_1S_\sigma(\alpha)+b_1>y_2S_\sigma(\alpha)+b_2$ and $y_1S_\sigma(\beta)+b_1=y_2S_\sigma(\beta)+b_2$ imply
	\begin{equation}\label{inequality 2}
	\forall\,\gamma\in [\alpha,\beta): y_1S_\sigma(\gamma)+b_1>y_2S_\sigma(\gamma)+b_2.
	\end{equation}
	
	In case of $\alpha_1\leq u_0$ it follows from the hypothesis that $H_{|[0,u_0]}$ is constant that trivially $yS+b\geq H$ on $[0,\alpha_1]$. Now let $u_0<\alpha_1$. We assume that $yS(\alpha)+b<H(\alpha)$ for some $\alpha\in(u_0,\alpha_1)$. Because $S_\sigma$ is strictly decreasing in $[u_0,1]$ there are $\tilde{y},\tilde{b}\in\mathbb{R}$ such that
	\[\tilde{y}S_\sigma(\alpha)+\tilde{b}=H(\alpha)\text{ and }\tilde{y}S(\alpha_2)+\tilde{b}=H(\alpha_2).\]
	The $S_\sigma$-concavity of $H$ hence implies $H\geq \tilde{y}S_\sigma+\tilde{b}$ on $[\alpha,\alpha_2]$. In particular
	\begin{equation}\label{inequality 3}
	yS_\sigma(\alpha_1)+b=H(\alpha_1)\geq \tilde{y}S_\sigma(\alpha_2)+\tilde{b}.
	\end{equation}
	On the other hand
	\[\tilde{y}S_\sigma(\alpha)+\tilde{b}=H(\alpha)>yS_\sigma(\alpha)+b\text{ and }\tilde{y}S_\sigma(\alpha_2)+\tilde{b}=H(\alpha_2)=yS_\sigma(\alpha_2)+b\]
	so that by~\eqref{inequality 2} we obtain $\tilde{y}S_\sigma+\tilde{b}>yS_\sigma+b$ on $[\alpha,\alpha_2)$ which contradicts~\eqref{inequality 3}. Therefore $yS+b\geq H$ on $(u_0,\alpha_1)$. Since $S_\sigma$ is continuous and $H$ is continuous from the right in $u_0$ the same inequality holds on $[u_0,\alpha_1)$. Because $S_\sigma$ and $H$ are constant on $[0,u_0]$ we obtain $yS+b\geq H$ on $[0,\alpha_1]$.
	
	It remains to show that $yS+b\geq H$ on $[\alpha_2,1]$ as well. Assume there is $\alpha\in[\alpha_2,1]$ with $yS_\sigma(\alpha)+b<H(\alpha)$. Since $\alpha_2>u_0$ there are again $\tilde{y},\tilde{b}\in\mathbb{R}$ such that
	\[\tilde{y}S_\sigma(\alpha_1)+\tilde{b}=H(\alpha_1)\text{ and }\tilde{y}S(\alpha)+\tilde{b}=H(\alpha).\]
	Because $H$ is $S_\sigma$-concave this implies
	\begin{equation}\label{inequality 4}
	\forall\,\beta\in [\alpha_1,\alpha]:\,H(\beta)\geq \tilde{y}S_\sigma(\beta)+\tilde{b}.
	\end{equation}
	On the other hand
	\[\tilde{y}S_\sigma(\alpha_1)+\tilde{b}=H(\alpha_1)=yS_\sigma(\alpha_1)+b\text{ and }\tilde{y}S_\sigma(\alpha)+\tilde{b}=H(\alpha)>yS_\sigma(\alpha)+b.\]
	By~\eqref{inequality 1} it follows $\tilde{y}S_\sigma+\tilde{b}>yS_\sigma+b$ on $(\max\{\alpha_1,u_0\},\alpha]$. In particular 
	\[\tilde{y}S_\sigma(\alpha_2)+\tilde{b}>yS_\sigma(\alpha_2)+b=H(\alpha_2)\]
	which contradicts~\eqref{inequality 4}.
\end{proof}

\begin{defn}
	Let $\sigma$ be a distortion function. For a continuous function $G:[0,1]\rightarrow\mathbb{R}$ we define
	\[G_\sigma^*:[0,\infty)\rightarrow\mathbb{R}, G_\sigma^*(y):=\inf_{\alpha\in[0,1]} yS_\sigma(\alpha)-G(\alpha).\]
	Then $G(\alpha)+G_\sigma^*(y)\leq yS_\sigma(\alpha)$ for all $\alpha\in[0,1], y\geq 0$ so that
	\[G_\sigma:[0,1]\rightarrow\mathbb{R}, G_\sigma(\alpha):=\inf_{y\geq 0}yS_\sigma(\alpha)-G_\sigma^*(y)\]
	is well-defined and satisfies $G_\sigma\geq G$.
\end{defn}

\begin{rem}\label{properties of G_sigma 1}~
	\begin{enumerate}
		\item Setting as before $u_0:=\inf\{u>0; \sigma(u)>0\}$ we have that $S_{\sigma|[u_0,1]}$ is a bijection from $[u_0,1]$ onto $[0,1]$. Denoting by abuse of notation its inverse with $S^{-1}$ it follows
		\[\forall\,\alpha\in [0,1]: G_\sigma(S^{-1}(\alpha))=\inf_{y\geq 0}y\alpha-G_\sigma^*(y)\]
		so that
		\[\tilde{G}_\sigma:[0,1]\rightarrow\mathbb{R}, \tilde{G}_\sigma(\alpha):=\inf_{y\geq 0}y\alpha-G_\sigma^*(y)\]
		is well-defined. Being the infimum of nondecreasing and concave functions $\tilde{G}_\sigma$ is nondecreasing and concave, too. Therefore, $\tilde{G}_\sigma$ is differentiable from the right on $[0,1)$ with non-negative and nonincreasing right derivative. We obviously have $G_\sigma=\tilde{G}_\sigma\circ S_\sigma$ so that the concavity of $S_\sigma$ implies that $G_\sigma$ is concave, too. Denoting left and right derivatives by $\frac{d_l}{d\alpha}$ and $\frac{d_r}{d\alpha}$ respectively, an appropriate adaption of your favorite proof of the chain rules yields
		\[\forall\, \alpha\in (0,1]:\,\frac{d_l}{d\alpha}G_\sigma(\alpha)=\frac{d_r}{d\alpha}\tilde{G}_\sigma(S_\sigma(\alpha))\frac{d_l}{d\alpha}S_\sigma(\alpha)=-\frac{d_r}{d\alpha}\tilde{G}_\sigma(S_\sigma(\alpha))\sigma(\alpha).\]
		Combined with $G_\sigma(1)-G_\sigma(\alpha)=\int_\alpha^1\frac{d_l}{d\alpha}G_\sigma(u)du$ we obtain
		\begin{equation}\label{chain rule}
		\forall\,\alpha\in [0,1]:\,G_\sigma(\alpha)=G_\sigma(1)+\int_\alpha^1 H(u)\sigma(u)du	
		\end{equation}
		for a non-negative, nondecreasing function $H$ on $[0,1]$ which is continuous from the left.
		\item For $y\geq 0$ the function $yS_\sigma-G_\sigma^*(y)$ is obviously $S_\sigma$-concave. It therefore follows from Proposition~\ref{concavity of infimum} that $G_\sigma$ is $S_\sigma$-concave. Moreover, being the infimum of nonincreasing functions $G_\sigma$ is nonincreasing.
		\item Because $S_{\sigma|[0,u_0]}=1$ it follows that $G_\sigma$ is constant on $[0,u_0]$. Hence, for all $\alpha_1,\alpha_2\in [0,1], \alpha_1<\alpha_2,$ there are $y\geq 0, b\geq G_\sigma^*(y)$ with $G_\sigma(\alpha_j)=yS_\sigma(\alpha_j)-b$ for $j=1,2$.
		
		Indeed, if $\alpha_2> u_0$ we choose $y=\frac{G_\sigma(\alpha_2)-G_\sigma(\alpha_1)}{S_\sigma(\alpha_2)-S_\sigma(\alpha_1)}$ which is well-defined and non-negative because $S_\sigma$ is strictly decreasing on $[u_0,1]$ and $G_\sigma$ is nonincreasing. Then
		\begin{align*}
		b&=yS_\sigma(\alpha_2)-G_\sigma(\alpha_2)\geq\inf_{\alpha\in[0,1]}yS_\sigma(\alpha)-G_\sigma(\alpha)\\
		&=\inf_{\alpha\in[0,1]}yS_\sigma(\alpha)-(\inf_{\tilde{y}\geq 0}\tilde{y}S_\sigma(\alpha)-G_\sigma^*(\tilde{y}))\geq G_\sigma^*(y).
		\end{align*}
		In case of $\alpha_2\leq u_0$ we may choose $y=0$ so that
		\[b=-G_\sigma(\alpha_2)=-G_\sigma(0)=-\inf_{y\geq 0}y-G_\sigma^*(y)\geq G_\sigma^*(0).\]
		If additionally $G$ is nonincreasing it holds
		\[G_\sigma(0)=\inf_{y\geq 0}y-G_\sigma^*(y)\leq -G_\sigma^*(0)=-\inf_{\alpha\in[0,1]}(-G(\alpha))=\sup_{\alpha\in[0,1]}G(\alpha)=G(0)\]
		so that because $G\leq G_\sigma$ we conclude
		\[b=-G_\sigma(0)=-G(0)=G_\sigma^*(0).\]
		
	\end{enumerate}
\end{rem}

\begin{prop}\label{properties of G_sigma 2}
	Let $G:[0,1]\rightarrow\mathbb{R}$ be continuous and nonincreasing. If $\alpha\in (0,1)$ is such that $G_\sigma(\alpha)>G(\alpha)$ then there are $0\leq\alpha_1<\alpha<\alpha_2\leq 1$ and $y\geq 0$ such that
	\[\forall\,\beta\in (\alpha_1,\alpha_2):\,G_\sigma(\beta)=yS_\sigma(\beta)-G_\sigma^*(y).\]
\end{prop}

\begin{proof}
	By continuity of $S_\sigma$ and $G_\sigma$ there are $0\leq\alpha_1<\alpha<\alpha_2\leq 1$ such that $G_\sigma(\alpha_2)>G(\alpha_1)$. From Remark~\ref{properties of G_sigma 1} iii) we conclude the existence of $y\geq 0$ and $b\geq G_\sigma^*(y)$ such that $G_\sigma(\alpha_j)=yS_\sigma(\alpha_j)-b, j=1,2$ and such that $y=0$ in case of $\alpha_2\leq u_0$.
	
	Because $S_\sigma$ and $G_\sigma$ are nonincreasing and $y\geq 0$ it follows from
	\[\inf_{\beta\in (\alpha_1,\alpha_2)}G_\sigma(\beta)=G_\sigma(\alpha_2)>G(\alpha_1)=\sup_{\beta\in (\alpha_1,\alpha_2)}G(\beta)\]
	that $yS_\sigma-b\geq G$ on $(\alpha_1,\alpha_2)$.
	
	If $\alpha_2\leq u_0$, we have seen in Remark~\ref{properties of G_sigma 1} iii) that without loss of generality we may assume $y=0$ and $b=-G(0)$. Since $G$ is nonincreasing it thus follows $yS_\sigma -b=G(0)\geq G$ on $[0,1]$.
	
	If $\alpha_2>u_0$ we apply Proposition~\ref{properties S and S-concave} to $G_\sigma$ to conclude $yS_\sigma-b\geq G_\sigma$ on $[0,1]\backslash (\alpha_1,\alpha_2)$. Since $G_\sigma\geq G$ and $yS_\sigma-b\geq G$ on $(\alpha_1,\alpha_2)$ we obtain also in this case $yS_\sigma-b\geq G$ on $[0,1]$.
	
	So in both cases $yS_\sigma-b\geq G$ or equivalently $yS_\sigma-G\geq b$ on $[0,1]$ so that $G_\sigma^*(y)=\inf_{\alpha}\,yS_\sigma(\alpha)-G(\alpha)\geq b$. Since also $b\geq G_\sigma^*(y)$ it follows $b=G_\sigma^*(y)$.
	
	Finally, since $G_\sigma$ is $S_\sigma$-concave and $G_\sigma(\alpha_j)=yS_\sigma(\alpha_j)-G_\sigma^*(\alpha_j)$ holds for $j=1,2$ it follows for $\beta\in (\alpha_1,\alpha_2)$
	\[yS_\sigma(\beta)-G_\sigma^*(y)\leq G_\sigma(\beta)=\inf_{\tilde{y}\geq 0}\tilde{y}S_\sigma(\beta)-G_\sigma^*(\tilde{y})\leq y S_\sigma(\beta)-G_\sigma^*(y)\]
	which proves the claim.
\end{proof}

\begin{prop}\label{properties G_sigma 3}
	Assume that $G:[0,1]\rightarrow\mathbb{R}$ is continuous, nonincreasing, and that $G(1)=0$. Then $G_\sigma(0)=G(0)$ and $G_\sigma(1)=0$.
\end{prop}

\begin{proof}
	We already observed in Remark~\ref{properties of G_sigma 1} iii) that $G_\sigma(0)=G(0)$. Using the compactness of $[0,1]$, $G(1)=0$, and that $S_\sigma(\alpha)=0$ implies $\alpha=1$ it follows 
	\[\forall\,n\in\mathbb{N}\,\exists\,k_n\in\mathbb{N}\,\forall\,\alpha\in [0,1]:\,G(\alpha)>k_n S_\sigma(\alpha)+\frac{1}{n}\]
	which implies that for every $n\in\mathbb{N}$ there is $k_n\in\mathbb{N}$ with $-1/n\leq G_\sigma^*(k_n)$. Because $S_\sigma(1)=0$ we derive
	\[G_\sigma(1)=\inf_{y\geq 0}(-G_\sigma^*(y))\leq\inf_{n\in\mathbb{N}}(-G_\sigma^*(k_n))\leq 0=G(1)\leq G_\sigma(1)\]
	which gives $G_\sigma(1)=0$.
\end{proof}

Combining Propositions~\ref{properties of G_sigma 2} and~\ref{properties G_sigma 3} we immediately obtain the next result.

\begin{prop}\label{properties of G_sigma 4}
	Let $G:[0,1]\rightarrow\mathbb{R}$ be continuous and nonincreasing such that $G(1)=0$. If $\alpha\in (0,1)$ satisfies $G_\sigma(\alpha)>G(\alpha)$ there are $0\leq \alpha_1<\alpha<\alpha_2\leq 1$ and $y\geq 0$ such that $G_\sigma=yS_\sigma-G_\sigma^*(y)$ on $(\alpha_1,\alpha_2)$ and $G_\sigma(\alpha_j)=G(\alpha_j)$ for $j=1,2$.
\end{prop}

We have now everything at hand to derive the analogue of Proposition~\ref{representation Koethe dual p=1}.

\begin{lem}\label{representation Koethe dual p>1}
	Let $\sigma$ be a distortion function and $p\in (1,\infty)$ with conjugate exponent $q$. Then $L_\sigma^p(P)^\times=L_{\sigma,q}^*(P)$, for every $Z\in L_\sigma^p(P)^\times$ it holds $|Z|_{\sigma,q}^*=\Vert\varphi_Z\Vert_{\sigma,p}^*$, and there is $Y\in L_\sigma^p(P)$ with $\Vert Y\Vert_{\sigma,p}=1$ such that $\varphi_Z(Y)=\Vert\varphi_Z\Vert_{\sigma,p}^*$.
\end{lem}

\begin{proof}
	By Proposition~\ref{one part of duality} we only have to show $L_\sigma^p(P)^\times\subseteq L_{\sigma,q}^*(P)$ and that
	\[\sup\{\mathbb{E}(ZY)|;\, Y\in L_\sigma^p(P), \Vert Y\Vert_{\sigma,p}\leq 1\}\]
	is an upper bound for $|Z|_{\sigma,q}^*$ for any $Z\in L_\sigma^p(P)^\times$.
	
	So we fix $Z\in L_\sigma^p(P)^\times$. By Remark~\ref{Koethe dual enthaelt betraege} we also have $|Z|\in L_\sigma^p(P)^\times$. We define
	\[G:[0,1]\rightarrow [0,\infty), G(\alpha):=\int_\alpha^1 F_{|Z|}^{-1}(u)du\]
	and observe that $G$ is well-defined by $|Z|\in L^1(P)$. $G$ is obviously continuous, differentiable from the left, nonincreasing with $G(1)=0$.
	By Proposition~\ref{properties of G_sigma 4} and Remark~\ref{properties of G_sigma 1} i) there is a non-negative, nondecreasing function $H$ on $[0,1]$ which is continuous from that left such that $G_\sigma(\alpha)=\int_\alpha^1 H(u)\sigma(u)du$.
	
	If there is $\alpha\in (0,1)$ with $G_\sigma(\alpha)>G(\alpha)$ it follows immediately from Proposition~\ref{properties of G_sigma 4} that there are $0\leq \alpha_1<\alpha<\alpha_2\leq 1$ and $y\geq 0$ such that $H(u)=y$ for $u\in (\alpha_1,\alpha_2)$ and $G_\sigma(\alpha_j)=G(\alpha_j), j=1,2$ so that
	\begin{align*}
	\int_{\alpha_1}^{\alpha_2} H(u)^q\sigma(u) du&=y^{q-1}\int_{\alpha_1}^{\alpha_2}H(u)\sigma(u)du=y^{q-1}(G_\sigma(\alpha_1)-G_\sigma(\alpha_2))\nonumber\\
	&=y^{q-1}(G(\alpha_1)-G(\alpha_2))=\int_{\alpha_1}^{\alpha_2}H(u)^{q-1}F_{|Z|}^{-1}(u)du.
	\end{align*}
	On the other hand, if $G_\sigma(\alpha)=G(\alpha)$ by continuity there is a maximal closed interval $[\alpha_1,\alpha_2]$ containing $\alpha$ such that $G$ and $G_\sigma$ coincide on $[\alpha_1,\alpha_2]$. Thus, on $(\alpha_1,\alpha_2)$ the left derivatives of $G$ and $G_\sigma$ coincide, i.e., $F_{|Z|}^{-1}=H\sigma$ on $(\alpha_1,\alpha_2)$ which implies again
	\[\int_{\alpha_1}^{\alpha_2}H(u)^q\sigma(u)du=\int_{\alpha_1}^{\alpha_2}H(u)^{q-1}F_{|Z|}^{-1}(u) du.\]
	Combining these arguments gives
	\begin{align}\label{equality of integrals}
	\forall\,\alpha_1,\alpha_2\in [0,1]:&\Big(G_\sigma(\alpha_1)=G(\alpha_1), G_\sigma(\alpha_2)=G(\alpha_2)\\
	&\;\;\Rightarrow\int_{\alpha_1}^{\alpha_2}H(u)^q\sigma(u)du=\int_{\alpha_1}^{\alpha_2}H(u)^{q-1}F_{|Z|}^{-1}(u)du.\Big)\nonumber
	\end{align}
	
	In order to proceed, we distinguish two cases. First we assume that there is a strictly increasing sequence $(\alpha_n)_{n\in\mathbb{N}}$ in $(0,1)$ converging to 1 such that $G(\alpha_n)=G_\sigma(\alpha_n)$. We define
	\[Y_n:=(\1_{[0,\alpha_n]} H^{q-1})\circ U,\]
	where $U\in\mathscr{U}(0,1)$ is coupled in a comonotone way with $|Z|$. From $1/p+1/q=1$ it follows that $|Y_n^p|=(\1_{[0,\alpha_n]} H^q)\circ U$ which implies
	\begin{equation}\label{truncated norm}
	\Vert Y_n\Vert_{\sigma,p}^p=\int_0^1 F_{|Y_n|}^{-1}(u)^p\sigma(u)du=\int_0^{\alpha_n}H^q(u)\sigma(u)du<\infty
	\end{equation}
	since $H$ in nondecreasing and $\int_0^1\sigma(u)du=1$ so that $Y_n=|Y_n|\in L_\sigma^p(P)$. Using the notation from Proposition~\ref{Koethe dual}, because $|Z|$ and $U$ are coupled in a comonotone way we have by~\eqref{truncated norm} and~\eqref{equality of integrals} applied to $\alpha_1=0$ and $\alpha_2=\alpha_n$
	\begin{align*}
	\int_0^{\alpha_n} H(u)^q\sigma(u)du&=\int_0^{\alpha_n} H^{q-1}(u)F_{|Z|}^{-1}(u)du=\mathbb{E}(|Y_n||Z|)\\
	&\leq\Vert\varphi_{|Z|}\Vert_{\sigma,p}^*\Vert Y_n\Vert_{\sigma,p}\nonumber=\Vert\varphi_{|Z|}\Vert_{\sigma,p}^*\Big(\int_0^{\alpha_n}H(u)^q\sigma(u)du\Big)^{1/p}
	\end{align*}
	which gives
	\[\Big(\int_0^1\1_{[0,\alpha_n]}H(u)^q\sigma(u)du\Big)^{1/q}\leq\Vert\varphi_{|Z|}\Vert_{\sigma,p}^*\]
	for all $n\in\mathbb{N}$. Using that $\lim_{n\rightarrow\infty}\alpha_n=1$ an application of the Monotone Convergence Theorem yields
	\begin{equation}\label{norm inequality}
	\Big(\int_0^1 H(u)^q\sigma(u)du\Big)^{1/q}\leq \Vert\varphi_{|Z|}\Vert_{\sigma,p}^*
	\end{equation}
	so that $Z^\prime:=H\circ U$ belongs to $L_\sigma^q(P)$. Because $Z^\prime=|Z^\prime|$ and $F_{|Z^\prime|}^{-1}=H$ it follows from
	\[\forall\,\alpha\in [0,1]:\,\int_\alpha^1 H(u)\sigma(u)du=G_\sigma(\alpha)\geq G(\alpha)=\int_\alpha^1 F_{|Z|}^{-1}(u)du\]
	that $Z^{\prime}{}_{\sigma}\negthickspace\succcurlyeq Z$ which combined with~\eqref{norm inequality} yields $Z\in L_{\sigma,q}^*$ and
	\[|Z|_{\sigma,q}^*\leq \Vert\varphi_{|Z|}\Vert_{\sigma,p}^*=\Vert\varphi_Z\Vert_{\sigma,p}^*\]
	where we have used Remark~\ref{Koethe dual enthaelt betraege} in the last equality. Since also $\Vert Z\Vert_{\sigma,p}^*\leq|Z|_{\sigma,q}^*$ we obtain from~\eqref{norm inequality}
	\begin{equation}\label{norm equality}
	\Vert \varphi_Z\Vert_{\sigma,p}^*=\Big(\int_0^1 H(u)^q\sigma(u)du\Big)^{1/q}=\inf\{\Vert Z^\prime\Vert_{\sigma,q};\,Z^{\prime}{}_{\sigma}\negthickspace\succcurlyeq Z\}.
	\end{equation}
	Now we define
	\begin{equation}\label{norm maximizer}
	Y:=\1_{\{Z\neq 0\}}\frac{\overline{Z}}{|Z|}\,H^{q-1}(U).
	\end{equation}
	Then the same arguments used in deriving~\eqref{truncated norm} combined with~\eqref{norm inequality} show that $Y\in L_\sigma^p(P)$ and
	\[\Vert Y\Vert_{\sigma,p}=\Big(\int_0^1 H^q(u)\sigma(u)du\Big)^{1/p}.\]
	Moreover, using that $|Z|$ and $U$ are coupled in a comonotone way, \eqref{equality of integrals} applied to $\alpha_1=0$ and $\alpha_2=1$, and~\eqref{norm equality} give
	\begin{align*}
	\varphi_Z(Y)&=\mathbb{E}(ZY)=\mathbb{E}(|ZY|)=\int_0^1 H(u)^{q-1}F_{|Z|}^{-1}(u)du\\
	&=\int_0^1 H(u)^q\sigma(u)du=\left(\int_0^1 H(u)^q\sigma(u)du\right)^{1/q}\left(\int_0^1 H(u)^q\sigma(u)du\right)^{1/p}\\
	&=\Vert\varphi_Z\Vert_{\sigma,p}^*\Vert Y\Vert_{\sigma,p}.
	\end{align*}
	
	Next, if there is no strictly increasing sequence $(\alpha_n)_{n\in\mathbb{N}}$ in $(0,1)$ converging to 1 such that $G(\alpha_n)=G_\sigma(\alpha_n)$ there is $\beta\in(0,1)$ such that $G(u)<G_\sigma(u)$ for all $\alpha\in (\beta,1)$ and such that $G(\beta)=G_\sigma(\beta)$. It therefore follows from Proposition~\ref{properties of G_sigma 4} that there is $y\geq 0$ such that $H=y$ on $(\beta,1)$. Because $H$ is nondecreasing this implies that $H$ is bounded so that trivially
	\[\int_0^1 H(u)^q\sigma(u)<\infty.\]
	By repeating the arguments from the first part of the proof it follows for $U\in\mathscr{U}(0,1)$ coupled in a comonotone way with $|Z|$ that $Z':=H\circ U$ satisfies $Z^\prime\in L_\sigma^q(P)$ and $Z^{\prime}{}_{\sigma}\negthickspace\succcurlyeq Z$ which gives $Z\in L_{\sigma,q}(P)^*$ and $|Z|_{\sigma,q}^*=\Vert\varphi_Z\Vert_{\sigma,p}^*=\Vert Z^\prime\Vert_{\sigma,q}$. Defining $Y$ as in~\eqref{norm maximizer} finally gives again $\varphi_Z(Y)=\Vert\varphi_Z\Vert_{\sigma,p}^*\Vert Y\Vert_{\sigma,p}$ which proves the claim.
\end{proof}

Combining Proposition~\ref{Koethe dual} and Lemma~\ref{representation Koethe dual p>1} we immediately derive the next result.

\begin{thm}\label{dual representation scalar case p>1}
	Let $\sigma$ be a distortion function and $p\in(1,\infty)$ with conjugate exponent $q$. Then $|\cdot|_{\sigma,q}^*$ is a norm on $L_{\sigma,q}^*(P)$ turning it into a Banach space. Moreover,
	\[\Phi:(L_{\sigma,q}^*(P),|\cdot|_{\sigma,q}^*)\rightarrow(L_\sigma^p(P)^*,\Vert\cdot\Vert_{\sigma,p}^*),\ Z\mapsto(Y\mapsto\Phi(Z)(Y):=\mathbb{E}(ZY))\]
	is an isometric isomorphism. Moreover, for every $\varphi\in L_\sigma^p(P)^*$ there is $Y\in L_\sigma^p(P)$ with $\Vert Y\Vert_{\sigma,p}=1$ such that $\varphi(Y)=\Vert\varphi\Vert_{\sigma,p}^*$.
\end{thm}

\begin{cor}\label{reflexivity}
	For a distortion function $\sigma$ and $p\in(1,\infty)$ the Banach space $L_\sigma^p(P)$ is reflexive.
\end{cor}

\begin{proof}
	This is an immediate consequence of James' Theorem (see, e.g.\ \citet[Theorem~I.3]{Diestel1975}) and Theorem~\ref{dual representation scalar case p>1}.
\end{proof}

\begin{prop}
	Simple functions (and thus $L^{\infty}$) are dense in $L_{\sigma,q}^*(P)$,
	whenever $q<\infty$.
\end{prop}
\begin{proof}
	Let $\mathfrak{F}$ contain all \emph{finite }sigma algebras $\mathcal{F}$
	for which the measure $P$ is defined. Note that $\left(\mathfrak{F},\subseteq\right)$
	is a filter, and the proof of Proposition~\ref{simple function dense} actually
	demonstrates that 
	\[
	\left\Vert \mathbb{E}\left(Y|\mathcal{F}\right)-Y\right\Vert _{\sigma,p}\xrightarrow[\mathfrak{F}]{}0
	\]
	whenever $\mathcal{F}\in\mathfrak{F}$ increases. 
	
	Recall first that $\AVaR_{\alpha}\bigl(\mathbb{E}(Y|\mathcal{F})\bigr)\le\AVaR_{\alpha}(Y)$.
	Indeed, it follows from the conditional Jensen inequality (cf.\ \citet[Section 34]{Williams1991})
	that $\bigl(\mathbb{E}(Y|\mathcal{F})-q\bigr)_{+}\le\mathbb{E}\bigl((Y-q)_{+}|\mathcal{F}\bigr),$
	and hence, using \citet{Pflug2000}, 
	\begin{align*}
	\AVaR_{\alpha}\bigl(\mathbb{E}(Y|\mathcal{F})\bigr) & =  \min_{q\in\mathbb{R}}q+\frac{1}{1-\alpha}\mathbb{E}\bigl(\mathbb{E}(Y|\mathcal{F})-q\bigr)_{+}\\
	&\le\min_{q\in\mathbb{R}}q+\frac{1}{1-\alpha}\mathbb{E}\mathbb{E}\bigl((Y-q)_{+}|\mathcal{F}\bigr)\\
	& = \min_{q\in\mathbb{R}}q+\frac{1}{1-\alpha}\mathbb{E}\bigl((Y-q)_{+}|\mathcal{F}\bigr)=\AVaR_{\alpha}(Y).
	\end{align*}
	
	Suppose that $Z^{\prime}\:{}_{\sigma}\negthickspace\succcurlyeq Z$.
	It follows that 
	\[
	\int_{\alpha}^{1}\sigma(u)F_{Z^{\prime}}^{-1}(u)\mathrm{d}u\ge\int_{\alpha}^{1}F_{Z}^{-1}(u)\mathrm{d}u\ge\int_{\alpha}^{1}F_{\mathbb{E}\left(Z|\mathcal{F}\right)}^{-1}(u)\mathrm{d}u
	\]
	for every $\alpha\le1$, that is $Z^{\prime}\:{}_{\sigma}\negthickspace\succcurlyeq\mathbb{E}\left(Z|\mathcal{F}\right)$
	and thus $\left\Vert \mathbb{E}\left(Z|\mathcal{F}\right)\right\Vert _{\sigma,q}^{*}\le\left\Vert Z\right\Vert _{\sigma,q}^{*}$.
	The assertion follows as $\left\{ \mathbb{E}\left(Z|\mathcal{F}\right)\colon\,\mathcal{F}\in\mathfrak{F}\right\} $
	is arbitrarily close to $Z$ in the norm $\left\Vert \cdot\right\Vert _{\sigma,q}$
	by Proposition~\ref{simple function dense}. 
\end{proof}

We close this section by having a closer look at $L_\sigma^1(P)$ and its dual space.



\begin{thm}
	The dual space of $L_\sigma^1(P)$ is not separable. 
\end{thm}

\begin{proof}
	It is enough to assume that $\sigma$ is unbounded, as for bounded $\sigma$ we have that $L_\sigma^1(P)$ is isomorphic to $L^1(P)$ by Proposition~\ref{bounded sigma} and its dual $L^{\infty}(P)$ is not separable.
	
	For $\beta\in[0,1]$ consider the random variables
	\[
	Z_{\beta}:=\begin{cases}
	\sigma(U) & \text{if }U\le\beta,\\
	\sigma(1+\beta-U) & \text{if }U>\beta
	\end{cases}
	\]
	for a (fixed) uniform random variable $U\in\mathscr{U}(0,1)$. Notice, that $\|Z_\beta\|_{\sigma,1}^*=1$, since $Z_\beta$ is a rearrangements of $\sigma(U)$.
	Assume that $\beta<\gamma$
	and observe that
	\begin{align*}
	Z_{\gamma}-Z_{\beta} & =\sigma(1+\gamma-U)-\sigma(1+\beta-U)\\
	& \ge\sigma(1+\gamma-U)-\sigma(1+\beta-\gamma)
	\end{align*}
	whenever $U>\gamma$. Then it holds that 
	\begin{align*}
	\left\Vert Z_{\gamma}-Z_{\beta}\right\Vert _{\sigma,1}^{*}
	&\ge\limsup_{\alpha\to1}\frac{\AVaR_{\alpha}\bigl(Z_\gamma-Z_\beta\bigr)}{\frac{1}{1-\alpha}\int_{\alpha}^{1}\sigma(u)\mathrm{d}u}\\
	&\ge\limsup_{\alpha\to1}\frac{\AVaR_{\alpha}\bigl(\sigma(1+\beta-U)-\sigma(1+\beta-\gamma)\bigr)}{\frac{1}{1-\alpha}\int_{\alpha}^{1}\sigma(u)\mathrm{d}u}.
	\end{align*}
	Now, as $\sigma$ is unbounded, the denominator is unbounded as well (indeed, we have $\frac{1}{1-\alpha}\int_{\alpha}^{1}\sigma(u)\mathrm{d}u\ge\sigma(\alpha)$)
	and hence
	\begin{align}
	\left\Vert Z_{\gamma}-Z_{\beta}\right\Vert _{\sigma,1}^{*} & \ge\limsup_{\alpha\to1}\frac{\AVaR_{\alpha}\bigl(\sigma(1+\beta-U)\bigr)-\sigma(1+\beta-\gamma)}{\frac{1}{1-\alpha}\int_{\alpha}^{1}\sigma(u)\mathrm{d}u}\nonumber \\
	& =\lim_{\alpha\to1}\frac{\frac{1}{1-\alpha}\int_{\alpha}^{1}\sigma(u)\mathrm{d}u-\sigma(1+\beta-\gamma)}{\frac{1}{1-\alpha}\int_{\alpha}^{1}\sigma(u)\mathrm{d}u}=1.\label{eq:11}
	\end{align}
	
	Suppose finally that there is a dense sequence $(D_{k})_{k\in\mathbb{N}}\subset L_\sigma^1(P)^*$. For $\beta\in[0,1]$ fixed there
	is $k\in\mathbb{N}$ such that $\left\Vert Z_{\beta}-D_{k}\right\Vert _{\sigma,1}^{*}<\frac{1}{2}$.
	But $1\le\left\Vert Z_{\beta}-Z_{\gamma}\right\Vert _{\sigma,1}^{*}\le\left\Vert Z_{\beta}-D_{k}\right\Vert _{\sigma,1}^{*}+\left\Vert D_{k}-Z_{\gamma}\right\Vert _{\sigma,1}^{*}$,
	from which follows that $\left\Vert D_{k}-Z_{\gamma}\right\Vert _{\sigma,1}^{*}>\frac{1}{2}$
	whenever $\gamma\neq\beta$. Hence only \emph{countably} many $Z_{\beta}$
	can be approximated by the sequence $(D_{k})_{k\in\mathbb{N}}$ with a distance $\left\Vert Z_{\beta}-D_{k}\right\Vert _{\sigma}^{*}<\frac{1}{2}$
	and $(D_{k})_{k\in\mathbb{N}}$ thus is not dense giving the desired contradiction. 
\end{proof}

\section{\label{sec:VectorDual}The dual space in the vector-valued case}

In this section we determine the dual space of $L_\sigma^p(P,X)$ for arbitrary Banach spaces $X$ over $\mathbb{K}\in\{\mathbb{R},\mathbb{C}\}$. We denote the space of $X$-valued simple functions on $(\Omega,\mathcal{F},P)$ by $\mathcal{S}(X)$, i.e.,
\[\mathcal{S}(X)=\{Y:\Omega\rightarrow X;\,Y(\Omega)\text{ is finite and }\forall\,x\in X:\,Y^{-1}(\{x\})\in\mathcal{F}\}.\]
Then it is straightforward to see and well-known that \[\{\varphi:\mathcal{S}(X)\rightarrow\mathbb{K};\,\varphi\text{ linear and continuous with respect to } \|\cdot\|_\infty\}\]
and
\[\{\mu:\mathcal{F}\rightarrow X^*;\,\mu\text{ vector measure of bounded variation}\}\] are isomorphic via the linear mapping
\begin{equation}\label{natural isomorphism}
\Phi:\varphi\mapsto\big(\mu_\varphi(E)(x):=\varphi(\1_E x),\,x\in X, E\in \mathcal{F}\big).
\end{equation}
For a vector measure $\mu$ we denote by $|\mu|$ its variation.

\begin{lem}\label{equality of suprema}
	For a linear mapping $\varphi:\mathcal{S}(X)\rightarrow\mathbb{K}$ we have
	\begin{align*}
	&\sup\{|\varphi(\sum_{j=1}^n\1_{E_j}x_j)|;\,E_j\in\mathcal{F}\text{ partition of }\Omega, x_j\in X, \Vert\sum_{j=1}^n\1_{E_j}x_j\Vert_{\sigma,p}\leq 1\}\\
	&=\sup\{\sum_{j=1}^n|\alpha_j|\,\Vert\mu_\varphi(E_j)\Vert;\,E_j\in\mathcal{F}\text{ partition of }\Omega, \alpha_j\in \mathbb{K}, \Vert\sum_{j=1}^n\alpha_j\1_{E_j}\Vert_{\sigma,p}\leq 1\}\\
	&=\sup\{\int_\Omega|\sum_{j=1}^n\alpha_j\1_{E_j}|\,d|\mu_\varphi|;\,E_j\in\mathcal{F}\text{ partition of }\Omega, \alpha_j\in \mathbb{K}, \Vert\sum_{j=1}^n\alpha_j\1_{E_j}\Vert_{\sigma,p}\leq 1\},
	\end{align*}
	where $\mu_\varphi$ is defined as in~\eqref{natural isomorphism}.
\end{lem}

\begin{proof}
	For a partition $E_1,\ldots,E_n\in\mathcal{F}$ of $\Omega$, $\alpha_1,\ldots,\alpha_n\in\mathbb{K}$, and $z_1,\ldots,z_n\in X$ with $\Vert z_j\Vert =1$ we have
	\[\Vert \sum_{j=1}^n\alpha_j\1_{E_j}z_j\Vert_{\sigma,p}^p=\sup_{U\in\mathcal{U}(0,1)}\int_\Omega\sigma(U)\sum_{j=1}^n|\alpha_j|^p\1_{E_j}\,dP=\Vert\sum_{j=1}^n\alpha_j,\1_{E_j}\Vert_{\sigma,p}^p,\]
	where the norm on the left hand side is the one on $L_\sigma^p(P,X)$ while the norm on the right hand side denotes the one on $L_\sigma^p(P)$. Therefore we conclude
	\begin{align*}
	&\sup\Big\{|\varphi(\sum_{j=1}^n\1_{E_j}x_j)|;\,E_j\in\mathcal{F}\text{ partition of }\Omega, x_j\in X, \Vert\sum_{j=1}^n\1_{E_j}x_j\Vert_{\sigma,p}\leq 1\Big\}\\
	&=\sup\Big\{|\varphi(\sum_{j=1}^n\alpha_j\1_{E_j}z_j)|;\,E_j\in\mathcal{F}\text{ partition of }\Omega, \alpha_j\geq 0, z_j\in X, \Vert z_j\Vert=1\\
	&\qquad\qquad\qquad\qquad\qquad\qquad\text{and } \Vert\sum_{j=1}^n\alpha_j\1_{E_j}\Vert_{\sigma,p}\leq 1\Big\}\\
	&=\sup\Big\{\sum_{j=1}^n|\varphi(\alpha_j\1_{E_j}z_j)|;\,E_j\in\mathcal{F}\text{ partition of }\Omega, \alpha_j\geq 0, z_j\in X, \Vert z_j\Vert=1\\
	&\qquad\qquad\qquad\qquad\qquad\qquad\text{and } \Vert\sum_{j=1}^n\alpha_j\1_{E_j}\Vert_{\sigma,p}\leq 1\Big\}\\
	&=\sup\Big\{\sum_{j=1}^n|\alpha_j|\,\Vert\mu_\varphi(E_j)\Vert;\,E_j\in\mathcal{F}\text{ partition of }\Omega, \alpha_j\in\mathbb{K}, \Vert\sum_{j=1}^n\alpha_j\1_{E_j}\Vert_{\sigma,p}\leq 1\Big\},
	\end{align*}
	which gives the first equality. Using the definition of $|\mu_\varphi|$ we continue
	\begin{align*}
	&\sup\Big\{\sum_{j=1}^n|\alpha_j|\,\Vert\mu_\varphi(E_j)\Vert;\,E_j\in\mathcal{F}\text{ partition of }\Omega, \alpha_j\in\mathbb{K}, \Vert\sum_{j=1}^n\alpha_j\1_{E_j}\Vert_{\sigma,p}\leq 1\Big\}\\
	&=\sup\Big\{\sum_{{i=1}\atop{\alpha_i\text{ pairwise different}}}^n|\alpha_i|\sum_{j:\alpha_i=\alpha_j}\Vert\mu_\varphi(E_j)\Vert;\,E_j\in\mathcal{F}\text{ partition of }\Omega, \alpha_j\in\mathbb{K}\\
	&\qquad\qquad\qquad\qquad\qquad\qquad\text{and } \Vert\sum_{j=1}^n\alpha_j\1_{E_j}\Vert_{\sigma,p}\leq 1\Big\}\\
	&=\sup\Big\{\sum_{{i=1}\atop{\alpha_i\text{ pairwise different}}}^n|\alpha_i||\mu_\varphi|\big(\cup_{j:\alpha_i=\alpha_j}E_j\big);\,E_j\in\mathcal{F}\text{ partition of }\Omega, \alpha_j\in\mathbb{K}\\
	&\qquad\qquad\qquad\qquad\qquad\qquad\text{and } \Vert\sum_{j=1}^n\alpha_j\1_{E_j}\Vert_{\sigma,p}\leq 1\Big\}\\
	&=\sup\Big\{\int_\Omega|\sum_{j=1}^n\alpha_j\1_{E_j}|\,d|\mu_\varphi|;\,E_j\in\mathcal{F}\text{ partition of }\Omega, \alpha_j\in\mathbb{K}, \Vert\sum_{j=1}^n\alpha_j\1_{E_j}\Vert_{\sigma,p}\leq 1\Big\}
	\end{align*}
	which proves the second equality.
\end{proof}

\begin{defn}
	For a distortion function $\sigma$, $p\in[1,\infty)$, and a Banach space $X$ we define
	\[\mathcal{L}_{\sigma,p}\big(\mathcal{S}(X)\big):=\{\varphi:\mathcal{S}(X)\rightarrow\mathbb{K};\,\varphi\text{ linear and continuous with respect to }\Vert\cdot\Vert_{\sigma,p}\}.\]
\end{defn}

\begin{lem}\label{auxiliary dual representation}
	Let $\Phi$ be the natural isomorphism from~\eqref{natural isomorphism}. Then $\Phi(\mathcal{L}_{\sigma,p}\big(\mathcal{S}(X)\big))$ coincides with the set
	\begin{align*}
	\{\mu:\mathcal{F}\rightarrow X^*&;\,\mu\text{ is a }\sigma\text{-additive vector measure of bounded variation}\\ 
	& \text{such that }|\mu|\ll P\text{ and }\frac{d|\mu|}{dP}\in L_{\sigma,q}^*(P)\},
	\end{align*}
	and
	\[\forall\,\varphi\in\mathcal{L}_{\sigma,p}\big(\mathcal{S}(X)\big):\,\Vert\varphi\Vert=|\frac{d |\mu_\varphi|}{dP}|_{\sigma,q}^*,\]
	where $q$ is the conjugate exponent to $p$.
\end{lem}

\begin{proof}
	For $\varphi\in\mathcal{L}_{\sigma,p}\big(\mathcal{S}(X)\big)$ it follows from the density of $\mathcal{S}(X)$ in $L_\sigma^p(P,X)$ that $\varphi$ extends to a unique element of $L_\sigma^p(P,X)^\times$ which we still denote by $\varphi$. For a pairwise disjoint sequence $(E_j)_{j\in\mathbb{N}}$ in $\mathcal{F}$ and its union $E$ it follows for arbitrary $x\in X$
	\begin{align*}
	|\mu_\varphi\big(E\big)(x)-\mu_\varphi\big(\bigcup_{j=1}^m E_j\big)(x)|&=|\mu_\varphi\big(\bigcup_{j=m+1}^\infty E_j\big)(x)|=|\varphi(\1_{\cup_{j=m+1}^\infty E_j}\,x)|\\
	&\leq \Vert\varphi\Vert_{\sigma,p}^*\,\Vert x\Vert\,\Vert \1_{\cup_{j=m+1}^\infty E_j}\Vert_{\sigma,p}\\
	&=\Vert\varphi\Vert_{\sigma,p}^*\,\Vert x\Vert\,\big(\int_{1-\sum_{j=m+1}^\infty P(E_j)}^1\sigma(u)du\big)^{1/p}.
	\end{align*}
	With the aid of Lebesgue's Dominated Convergence Theorem it follows
	\[\Vert \mu_\varphi(E)-\sum_{j=1}^m\mu_\varphi(E_j)\Vert\leq\Vert\varphi\Vert_{\sigma,p}^*\big(\int_{1-\sum_{j=m+1}^\infty P(E_j)}^1\sigma(u)du\big)^{1/p}\rightarrow_{m\rightarrow\infty}0.\]
	Thus $\Phi(\varphi)=\mu_\varphi$ is a $\sigma$-additive vector measure. Moreover, for every finite partition $E_1,\ldots,E_n\in\mathcal{F}$ of $\Omega$ and $x_1,\ldots,x_n\in X$ with $\Vert x_j\Vert\leq 1$ we have
	\begin{align*}
	\sum_{j=1}^n|\mu_\varphi\big(E_j\big)(x_j)|&=\sum_{j=1}^n|\varphi(\1_{E_j}x_j)|=\sum_{j=1}^n\sign\big(\varphi(\1_{E_j}x_j)\big)\,\varphi(\1_{E_j}x_j)\\
	&=|\varphi(\sum_{j=1}^n\sign\big(\varphi(\1_{E_j}x_j)\big)\1_{E_j}x_j)|\\
	&\leq\Vert\varphi\Vert_{\sigma,p}^*\,\Vert \sum_{j=1}^n\sign\big(\varphi(\1_{E_j}x_j)\big)\1_{E_j}x_j\Vert_{\sigma,p}\\
	&\leq\Vert\varphi\Vert_{\sigma,p}^*\,\big(\sup_{U\in\mathscr{U}(0,1)}\int_\Omega\sum_{j=1}^n\sign\big(\varphi(\1_{E_j}x_j)\big)^p\1_{E_j}\Vert x_j\Vert^p\sigma(U)dP\big)^{1/p}\\
	&\leq\Vert\varphi\Vert_{\sigma,p}^*\,\big(\sup_{U\in\mathscr{U}(0,1)}\int_\Omega\sigma(U)dP\big)^{1/p}=\Vert\varphi\Vert_{\sigma,p}^*,
	\end{align*}
	where for a complex number $\alpha$ as usual $\sign(\alpha)=\frac{\overline{\alpha}}{|\alpha|}$ in case $\alpha\neq 0$, resp.\ $\sign(0)=0$. Thus, for arbitrary $\varepsilon>0$ it follows for suitable choices $x_j^\varepsilon\in X$ from the above inequality that
	\[\sum_{j=1}^n\Vert\mu_\varphi(E_j)\Vert\leq\sum_{j=1}^n\big(|\mu_\varphi\big(E_j\big)(x_j^\varepsilon)|+\frac{\varepsilon}{n}\big)\leq\Vert\varphi\Vert_{\sigma,p}^*+\varepsilon,\]
	i.e., $\sum_{j=1}^n\Vert\mu_\varphi(E_j)\Vert\leq\Vert\varphi\Vert_{\sigma,p}^*$ which in turn implies $|\mu_\varphi|(\Omega)\leq\Vert\varphi\Vert_{\sigma,p}^*$. Hence $\Phi(\varphi)=\mu_\varphi$ is of bounded variation.
	
	Since $\mu_\varphi$ is $\sigma$-additive the same holds for $|\mu_\varphi|$ (see \citet[Proposition I.1.9]{Diestel1977}), i.e., $|\mu_\varphi|$ is a (finite) measure on $\mathcal{F}$. If $E\in\mathcal{F}$ satisfies $P(E)=0$ it follows for $x\in X$
	\[\Vert \1_E x\Vert_{\sigma,p}=\Vert x\Vert (\sup_{u\in\mathscr{U}(0,1)}\int_E\sigma(U)dP)^{1/p}=0\]
	and therefore $\Vert\mu_\varphi(E)\Vert=0$. If $E_1,\ldots,E_n\in \mathcal{F}$ is a partition of $E$ it follows $P(E_j)=0$ and thus $\sum_{j=1}^n \Vert\mu_\varphi(E_j)\Vert=0$ which implies $|\mu_\varphi|(E)=0$. By an application of the Radon-Nikod\'ym Theorem we obtain $g_\varphi\in L^1(P), g_\varphi\geq 0$ such that
	\[\forall\, E\in\mathcal{F}:\,\int_E g_\varphi\,dP=|\mu_\varphi|(E).\]
	From the fact that $\mathcal{S}(X)$ is dense in $L_\sigma^p(P,X)$ and $\mathcal{S}(\mathbb{K})$ is dense in $L_\sigma^p(P)$ it follows with Lemma~\ref{auxiliary dual representation}
	\begin{align*}
	\Vert\varphi\Vert_{\sigma,p}^*&=\sup\Big\{|\varphi(\sum_{j=1}^n\1_{E_j}x_j)|;\,E_j\in\mathcal{F}\text{ partition of }\Omega, x_j\in X, \Vert\sum_{j=1}^n\1_{E_j}x_j\Vert_{\sigma,p}\leq 1\Big\}\\
	&=\sup\Big\{\int_\Omega|\sum_{j=1}^n\alpha_j\1_{E_j}|\,d|\mu_\varphi|;\,E_j\in\mathcal{F}\text{ partition of }\Omega, \alpha_j\in \mathbb{K}\\
	&\qquad\qquad\qquad\qquad\qquad\qquad\text{and } \Vert\sum_{j=1}^n\alpha_j\1_{E_j}\Vert_{\sigma,p}\leq 1\}\\
	&=\sup\Big\{\int_\Omega|\sum_{j=1}^n\alpha_j\1_{E_j}|g_\varphi\,dP;\,E_j\in\mathcal{F}\text{ partition of }\Omega, \alpha_j\in \mathbb{K}\\
	&\qquad\qquad\qquad\qquad\qquad\qquad\text{and } \Vert\sum_{j=1}^n\alpha_j\1_{E_j}\Vert_{\sigma,p}\leq 1\Big\}\\
	&=\sup\Big\{\int_\Omega|f g_\varphi|\,dP;f\in L_\sigma^p(P), \Vert f\Vert_{\sigma,p}\leq 1\Big\},
	\end{align*}
	so that in particular $g_\varphi\in L_\sigma^p(P)^\times=L_{\sigma,q}^*(P)$ and $\Vert\varphi\Vert_{\sigma,p}^*=|g_\varphi|_{\sigma,q}^*$. Since $\varphi\in\mathcal{L}_{\sigma,p}(\mathcal{S}(X))$ was chosen arbitrarily this finally shows that $\Phi(\mathcal{L}_{\sigma,p}(\mathcal{S}(X))$ is contained in the set of $X^*$-valued, $\sigma$-additive vector measures of bounded variation such that their bounded variation measure admits a $P$-density in $L_{\sigma,q}^*(P)$.
	
	Next let $\mu$ be such a measure and set $\varphi:=\Phi^{-1}(\mu)$. We have to show that $\varphi$ belongs to $\mathcal{L}_{\sigma,p}\big(\mathcal{S}(X)\big)$. But from the density of $\mathcal{S}(\mathbb{K})$ in $L_\sigma^p(P)$ it follows immediately together with Lemma~\ref{equality of suprema} that
	\begin{align*}
	&\sup\{|\varphi(\sum_{j=1}^n\1_{E_j}x_j)|;\,E_j\in\mathcal{F}\text{ partition of }\Omega, x_j\in X, \Vert\sum_{j=1}^n\1_{E_j}x_j\Vert_{\sigma,p}\leq 1\}\\
	&=\sup\{\int_\Omega|\sum_{j=1}^n\alpha_j\1_{E_j}|\,\frac{d|\mu|}{dP}\,dP;\,E_j\in\mathcal{F}\text{ partition of }\Omega, \alpha_j\in \mathbb{K}, \Vert\sum_{j=1}^n\alpha_j\1_{E_j}\Vert_{\sigma,p}\leq 1\}\\
	&=|\frac{d|\mu|}{dP}|_{\sigma,p}^*<\infty,
	\end{align*}
	which shows $\varphi\in\mathcal{L}_{\sigma,p}\big(\mathcal{S}(X)\big)$.
\end{proof}

\begin{defn}
	Let $X$ be a Banach space, $\sigma$ a distortion function, and $p\in [1,\infty)$ with conjugate exponent $q$. Then we define
	\begin{align*}
	L_{\sigma,q}^*(P,X^*):=\{\mu:\mathcal{F}\rightarrow X^*;&\ \mu\text{ is a }\sigma\text{-additive vector measure of bounded}\\ 
	& \text{variation such that }|\mu|\ll P\text{ and }\frac{d|\mu|}{dP}\in L_{\sigma,q}^*(P)\},
	\end{align*}
	which is obviously a subspace of the space of all $X^*$-valued vector measures on $\mathcal{F}$. Moreover, for $\mu\in L_{\sigma,q}^*(P,X^*)$ we set $|\mu|_{\sigma,q}^*:=|\frac{d|\mu|}{dP}|_{\sigma,q}^*$. Then, $|\cdot|_{\sigma,q}^*$ is obviously a norm on $L_{\sigma,q}^*(P,X^*)$.
\end{defn}

\begin{rem}\label{integration with respect to a vector measure}
	For $\mu\in L_{\sigma,q}^*(P,X^*)$ it follows from Lemma~\ref{auxiliary dual representation} and the density of $\mathcal{S}(X)$ in $L_\sigma^p(P,X)$ that $\Phi^{-1}(\mu)$ can be extended in a unique way to a continuous linear functional on $L_\sigma^p(P,X)$ which we again denote by $\Phi^{-1}(\mu)$. For $Y\in L_\sigma^p(P,X)$ we also write for obvious reasons
	\[\int_\Omega Yd\mu:=\Phi^{-1}(\mu)(Y).\]
	With this notation the following theorem is an immediate consequence of Lemma~\ref{auxiliary dual representation}, Proposition~\ref{representation Koethe dual p=1}, and Lemma~\ref{representation Koethe dual p>1}.
\end{rem}

\begin{thm}\label{dual representation}
	Let $X$ be a Banach space, $\sigma$ a distortion function, and $p\in [1,\infty)$ with conjugate exponent $q$. Then $(L_{\sigma,q}^*(P,X^*),|\cdot|_{\sigma,q}^*)$ is a Banach space and the mapping
	\[\Psi:(L_{\sigma,q}^*(P,X^*),|\cdot|_{\sigma,q}^*)\mapsto (L_\sigma^p(P,X)^*,\Vert\cdot\Vert_{\sigma,p}^*),\ \mu\mapsto(Y\mapsto\int_\Omega Yd\mu)\]
	is an isometric isomorphism.
\end{thm}

\begin{defn}\label{vector-valued Koethe dual}
	For a Banach space $X$, $p\in [1,\infty)$ with conjugate exponent $q$ we define
	\[L_\sigma^{q\,*}(P,X):=\{Z:\Omega\rightarrow X;\, Z\text{ strongly measurable, }\Vert Z\Vert\in L_{\sigma,q}^*(P)\}\]
	and for $Z\in L_\sigma^{q\,*}(P,X)$ we set $|Z|_\sigma^{q,*}:=|\,\Vert Z\Vert\,|_{\sigma,q}^*$, where as usual we identify random variables which coincide $P$-almost everywhere. It follows easily that $L_\sigma^{q\,*}(P,X)$ is a vector space and $|\cdot|_\sigma^{q,*}$ a norm.
\end{defn}

\begin{rem}\label{vector-valued Koethe dual isometrically contained}
	For $Z\in L_\sigma^{q\,*}(P,X^*)$ it follows from $\Vert Z\Vert\in L_{\sigma,q}^*(P)\subseteq L^1(P)$ that
	\[\mu_Z:\mathcal{F}\rightarrow X^*,\mu_Z(E):=\int_E Z\,dP\]
	is a well-defined, $\sigma$-additive vector measure of bounded variation with $|\mu_Z|(E)=\int_E\Vert Z\Vert\,dP$ (see, e.g., \cite[Theorem II.2.4]{Diestel1977}). A straightforward calculation gives for $Y\in \mathcal{S}(X)$
	\[\int_\Omega Y\,d\mu_Z=\int_\Omega\langle Z(\omega),Y(\omega)\rangle\,dP(\omega).\]
	Moreover, for $Z\in L_\sigma^{q\,*}(P,X^*)$ and $Y\in L_\sigma^p(P,X)$ it follows from $\Vert Z\Vert\in L_{\sigma,q}^*(P)$ and $\Vert Y\Vert\in L_\sigma^p(P)$
	\begin{align*}
	\int_\Omega |\langle Z(\omega),Y(\omega)\rangle|\,dP(\omega)&\leq \int_\Omega \Vert Z(\omega)\Vert\,\Vert Y(\omega)\Vert\,dP(\omega)\leq|\,\Vert Z\Vert\,|_{\sigma,q}^*\,|\,\Vert Y\Vert\,|_{\sigma,p}\\
	&=|Z|_\sigma^{q\,*}\,\Vert Y\Vert_{\sigma,p}
	\end{align*}
	which implies that $\psi_Z:L_\sigma^p(P,X)\rightarrow \mathbb{K}, Y\mapsto\mathbb{E}(\langle Z,Y\rangle)$ is a well-defined continuous linear functional which coincides on the dense subspace $\mathcal{S}(X)$ with $\Psi(\mu_Z)$. Together with Theorem~\ref{dual representation} this shows that
	\[\iota:(L_\sigma^{q\,*}(P,X^*),|\cdot|_\sigma^{q,*})\rightarrow L_\sigma^p(P,X)^*,\ Z\mapsto \Big(Y\mapsto\mathbb{E}(\langle Z,Y\rangle)\Big)\]
	is an isometry.
\end{rem}

As in the case of Bochner-Lebesgue spaces we have the following result.

\begin{thm}\label{Radon Nikodym property}
	For a Banach space $X$, a distortion function $\sigma$, and $p\in [1,\infty)$ with conjugate exponent $q$ the isometry
	\[\iota:(L_\sigma^{q\,*}(P,X^*),|\cdot|_\sigma^{q,*})\rightarrow L_\sigma^p(P,X)^*,\ Z\mapsto \Big(Y\mapsto\mathbb{E}(\langle Z,Y\rangle)\Big)\]
	is an isomorphism if and only if $X^*$ has the Radon-Nikod\'ym property with respect to $(\Omega,\mathcal{F},P)$.
\end{thm}

\begin{proof}
	Assume first, that $X^*$ has the Radon-Nikod\'ym property with respect to $(\Omega,\mathcal{F},P)$. By Remark~\ref{vector-valued Koethe dual isometrically contained} we only have to show surjectivity of $\iota$. For an arbitrary $\varphi\in L_\sigma^p(P,X)^*$ there is by Theorem~\ref{dual representation} a $\sigma$-additive $X^*$-valued vector measure of bounded variation such that $|\mu|\ll P$ and $\frac{d|\mu|}{dP}\in L_{\sigma,q}^*(P)$ with $\Vert \varphi\Vert=|\frac{d|\mu|}{dP}|_{\sigma,q}^*$. By the Radon-Nikod\'ym property of $X^*$ it follows that there is $Z\in L^1(P,X^*)$ such that $\mu(E)=\int_E Z\,dP$ for all $E\in\mathcal{F}$. Since $|\mu|(E)=\int_E\Vert Z\Vert\,dP$ (see e.g.\ \cite[Theorem II.2.4]{Diestel1977}) it follows that $Z\in L_\sigma^{q\,*}(P,X^*)$ and $\iota(Z)=\mu$ showing the surjectivity of $\iota$.
	
	Now, let $\iota$ be an isometric isomorphism. The proof that $X^*$ has the Radon-Nikod\'ym property is along the same lines as the proof of the corresponding implication of \cite[Theorem IV.1.1]{Diestel1977}. However, we include the proof for the reader's convenience. So, let $\mu:\mathcal{F}\rightarrow X^*$ be a $P$-continuous vector measure of bounded variation and fix $E_0\in\mathcal{F}$ such that $P(E_0)>0$. By the Hahn Decomposition Theorem applied to the signed measure $kP-|\mu|$ for large enough $k>0$ gives the existence of $B\in\mathcal{F}, B\subseteq E_0, P(B)>0$ such that $|\mu|(E)\leq kP(E)$ for all $E\in\mathcal{F}, E\subseteq B$. For $Y\in\mathcal{S}(X), Y=\sum_{j=1}^n\1_{E_j}x_j$ with pairwise disjoint $E_j\in\mathcal{F}$ and $x_j\in X$ we define
	\[\varphi(Y)=\sum_{j=1}^n\mu(E_j\cap B)(x_j).\]
	Denoting the norm in $L^1(P,X)$ as usual by $\Vert\cdot\Vert_1$ Theorem~\ref{Banach space} then gives
	\begin{equation*}
	|\varphi(Y)|\leq\sum_{j=1}^n k\Vert\mu(E_j\cap B)(x_j)\Vert\leq k\Vert Y\Vert_1\leq k\Vert Y\Vert_{\sigma,p}
	\end{equation*}
	so that the obviously linear mapping $\varphi$ on $\mathcal{S}(X)$ is continuous with respect to $\Vert\cdot\Vert_{\sigma,p}$. By Proposition~\ref{simple function dense} $\varphi$ extends (in a unique way) to an element of $L_\sigma^p(P,X)^*$ which we still denote by $\varphi$. Since $\iota$ is supposed to be surjective there is $Z\in L_\sigma^{q\,*}(P,X^*)\subseteq L^1(P,X^*)$ such that
	\[\forall\,Y\in L_\sigma^p(P,X):\,\varphi(Y)=\mathbb{E}(\langle Z,Y\rangle).\]
	Since $\mu(E\cap B)(x)=\varphi(\1_E x)=\int_E\langle Z(\omega), x\rangle\,dP(\omega)=\langle\int_E Z(\omega)\,dP(\omega),x\rangle$ for all $E\in\mathcal{F}, x\in X$ it follows that $\mu(E\cap B)=\int_E Z\,dP$.
	
	Because $E_0\in\mathcal{F}$ with $P(E_0)>0$ was chosen arbitrarily, it follows from \cite[Corollary III.2.5]{Diestel1977} that there is $Z\in L^1(P,X^*)$ such that $\mu(E)=\int_E Z\,dP$ for all $E\in\mathcal{F}$ which proves the Radon-Nikod\'ym property of $X^*$ with respect to $(\Omega,\mathcal{F},P)$.
\end{proof}

\section{Summary}

This paper introduces Banach spaces, which naturally carry risk measures
for vector-valued returns. Risk measures are continuous on these spaces,
and the spaces are as large as possible. The spaces are built based
on duality, and in this sense are natural for risk measures involving
vector-valued returns. We provide a complete characterization of the topological
dual, which essentially simplifies if the dual of the state space enjoys the Radon--Nikodým
property.

It is a key property of these spaces that the corresponding risk functional
is continuous (in fact, Lipschitz continuous) with respect to any
of the associated norms introduced, such that they all qualify as
a domain space for the risk measure.

\bibliographystyle{abbrvnat}
\bibliography{LiteraturAlois}

\end{document}